\documentclass{amsart}

\usepackage{amsfonts}
\usepackage{amsmath}
\usepackage{amssymb}
\usepackage{graphicx}
\usepackage{flafter}
\usepackage{bm}
\usepackage[all]{xy}
\usepackage[left,modulo]{lineno}

\usepackage{multirow}
\usepackage{lscape}
\usepackage{mathrsfs}
\graphicspath{ {Figures/} }

\theoremstyle{plain}
\newtheorem{theorem}{Theorem}[section]
\newtheorem{claim}{Claim}[section]

\newtheorem{corollary}[theorem]{Corollary}

\newtheorem{definition}[theorem]{Definition}

\newtheorem{lemma}[theorem]{Lemma}
\newtheorem{notation}[theorem]{Notation}
\newtheorem{proposition}[theorem]{Proposition}

\newtheorem{remark}[theorem]{Remark}
\newtheorem{theorem-definition}[theorem]{Theorem and Definition}

\numberwithin{equation}{section}

\begin{document}

\title{Some Banach-Lie structures}

\date{}

\title{ An integrability criterion for a projective limit of Banach distributions}
\author{Fernand Pelletier}

\address{Unit\'e Mixte de Recherche 5127 CNRS, Universit\'e  de Savoie Mont Blanc, Laboratoire de Math\'ematiques (LAMA),Campus Scientifique,  73370 Le Bourget-du-Lac, France}
\email{fernand.pelletier@univ-smb.fr}

\date{}

\begin{abstract}
 We give an integrability criterion  for  a projective limit of Banach distributions on a Fr\'echet manifold which is a projective limit of Banach manifolds.  This leads to a result of integrability of projective limit of involutive bundles on  a projective sequence of Banach manifolds. This   can be seen as a version of Frobenius Theorem in Fr\'echet  setting.   As consequence,    we obtain a  version of the third Lie theorem    for a Fr\'echet-Lie group which is a   submersive projective limit of Banach Lie groups.  We also give an application to a sequence of prolongations of a Banach Lie algebroid.
 \end{abstract}
 \maketitle
 
 \section{introduction}
In classical differential geometry, a \textit{distribution} on a smooth
manifold $M$, is an assignment
$\Delta :x\mapsto\Delta_{x}\subset T_{x}M$
on $M$, where $\Delta_{x}$ is a subspace of $T_{x}M$. This distribution
is \textit{integrable} if, for any $x\in M$, there exists an immersed
submanifold $f:L\rightarrow M$ such that $x\in f(L)$ and for any $z\in L$, we
have $Tf(T_{z}L)=\Delta_{f(z)}$. On the other hand, $\Delta$ is
called \textit{involutive} if, for any vector fields $X$ and $Y$ on $M$
tangent to $\Delta$, their Lie bracket $[X,Y]$ is also tangent to
$\Delta$.

On a finite dimensional manifold, when $\Delta$ is a subbundle of $TM$, the
classical Frobenius Theorem gives an equivalence between integrability and
involutivity. In the other cases, the distribution is \textit{singular} and,
even under assumptions of smoothness on $\Delta$, in general, the
involutivity is not a sufficient condition for integrability (one needs some
more additional local conditions). These problems were clarified and resolved
essentially in \cite{Sus} and \cite{Ste}.

In the context of Banach manifolds, the Frobenius Theorem is again true for
distributions which are complemented subbundles in the tangent bundle (cf. \cite{Lan}). For
singular Banach distributions closed and complemented (i.e. $\Delta_{x}$
is a complemented Banach subspace of $T_{x}M$) we also have the integrability
property under some natural geometrical conditions (see \cite{ChSt} for
instance). In a more general way, weak Banach distributions $\Delta$
(i.e. $\Delta_{x}$ which is a Banach subspace   of  $T_{x}M$ not necessary complemented), the integrability property is again true under some additional geometrical assumptions (see \cite{Pe} or \cite{CaPe} for more details).\\

 The proof of this last  results is essentially  based on the existence of the flow of a  local vector field. In a more general infinite dimensional context as  distributions on convenient manifolds  or on locally convex manifolds, in general the local  flow for a vector field does not exist. Analog results exists in such previous settings: \cite{GlLo},  \cite{Hil}, \cite{Eyn}, \cite{Tei} \cite{CaPe1} for instance. But essentially,  all these integrability criteria  are proved under  strong assumptions which,  either  implies    the existence of a  family of  vector fields which are tangent and generate  locally a distribution and each one of these vector fields have a local flow,  or  implies   the  existence of an implicit function theorem in such a setting.\\
 
 The purpose of this paper is to give an integrability criterion for projective limit of Banach distributions on  a Fr\'echet manifold which is a projective limit of Banach manifolds. The precise assumptions on the distribution  are presented  in  assumptions (*) in Definition \ref{D_SubmersiveProjectiveLimitOfLocalBanachBundles}) and this criterion  is formulated  in a local way in  Theorem \ref {T_IntegrabilitySPLBanachBundles} and in a global way in Theorem \ref{T_WeakDistribution}. These  results are obtained under conditions  which permits to used a Theorem of existence and unicity of solution of ODE in a Fr\'echet space proved in \cite{Lob}, and   which can be reformulated   in the context of projective limit of Banach spaces (cf. Appendix \ref{ExistenceODE}). Using  such a Theorem, this  proof  needs, in the one hand, an adaptation of some arguments used in the proof  of Theorem 1 in \cite{Pe} for closed distributions on Banach manifolds, and on the other hand, some properties of the Banach Lie  group of "uniformly bounded" automorphisms of a Fr\'echet space (cf.  Appendix \ref{__TheFrechetSpaceHF1F2-TheBanachSpaceHbF1F2}). As application, this criterion permits to obtain a kind of {\it  projective limit of "Banach Frobenius theorem" } for submersive  projective limit of involutive bundles on  a submersive  projective sequence of Banach manifolds (cf. Theorem \ref{T_Frobenius}).  
 By the way,  as consequence, a submersive  projective limit of complemented Banach  Lie subalgebras of  a submersive projective limit of Banach Lie group algebras is the Lie algebra of a F\'echet Lie group (cf. Theorem \ref{T_IntegrableFrechet Group}).  This result can be considered as {\it a version of the third Lie Theorem for a Fr\'echet-Lie group } which is a   submersive projective limit of Banach Lie groups. We also give an application to a sequence of prolongations of a Banach Lie algebroid (cf. Theorem \ref{T_SequenceProlongation}). We complete  these results, by  an example of  integrable Fr\'echet distribution which is a projective limit of non integrable distributions but  which satisfies the  assumptions (*).\\

 This paper is organized as follows. The first paragraph of the next section  describes the context and   the assumptions of this criterion of integrability  used in Theorem \ref{T_IntegrabilitySPLBanachBundles} and Theorem \ref{T_WeakDistribution}. In order  to present  these  Theorems  in a  more accessible way  for  a quick
 reading, the useful  definitions and results take place  in Appendix \ref {_ProjectiveLimits}, and   we formulate the assumptions   with precise references to this Appendix. The  Theorem \ref{T_WeakDistribution} permits to show that, under some natural conditions,  the projective limit  $E$ of a  submersive projective sequence of involutive subbundles  $ E_i$ of the tangent bundle $TM_i$ of a submersive projective sequence of Banach manifold $M_i$, is a Fr\'echet involutive and integrable  subbundle of the tangent bundle  $TM$  of the Fr\'echet manifold $M=\underleftarrow{\lim}M_i$. \\  A first application of these results, is the existence of a F\'echet Lie group whose Lie algebra is    a submersive projective limit of  complemented Banach  Lie subalgebra of the  a submersive  projective limit of Banach Lie groups Lie algebras  (cf. Theorem \ref{T_IntegrableFrechet Group} in $\S$ \ref{Application1}). In $\S$ \ref{application2},  we give an application of Theorem \ref{T_IntegrabilitySPLBanachBundles} to a sequence of 
 prolongations of a Banach-Lie algebroid (see \cite{CaPe2}) and we end this paragraph by the  announced contre-exemple of Theorem \ref{T_IntegrabilitySPLBanachBundles} and Theorem \ref{T_WeakDistribution}. \\
 The proof of the basic  Theorem    \ref{T_IntegrabilitySPLBanachBundles}
 is located in $\S$\ref{ProofIntegrabilitySPLBanachBundles}. \\
 All  properties  concerning the set    of uniformly bounded endomorphisms of a 
 Fr\'echet space are developed in Appendix \ref{__TheFrechetSpaceHF1F2-TheBanachSpaceHbF1F2}.  Also in a series of other Appendices,  we expose all the definitions and   results needed in  the statements of Theorems  and in the proof of Theorem \ref{T_IntegrabilitySPLBanachBundles}.

\section{An integrability  criterion  for a submersive projective limit of Banach distributions}
\label{_ACriterionOfIntegrabilityForProjectiveLimitOfBanachDistributions}
\subsection{The criterion and its corollaries}\label{__TheCriterion}

\emph{The context needed in this section is detailed in Appendix \ref{_ProjectiveLimits}.}\\
Let $\left(M_i,\delta_i^j\right)_{j\geq i}$ be a sequence of projective Banach manifolds with projective limit $M= \underleftarrow{\lim}M_i$ \footnote{ cf. section \ref{__ProjectiveLimitsOfBanachManifolds}}.
In order to give a criterion of integrability for projective limits of local Banach bundles on $M$ under some additional assumptions, we need  to introduce some notations.\\
Let $\nu$ (resp.  $\mu$) be a norm on  a Banach space $\mathbb{E}$ (resp. $\mathbb{M}$). We denote by $\|\;\|^{\operatorname{op}}$  the associated norm on the linear space of continuous linear mappings $\mathcal{L}(\mathbb{E},\mathbb{M})$.\\
Then we have: 
\begin{definition}
\label{D_SubmersiveProjectiveLimitOfLocalBanachBundles}
Let $\left( M_i,\delta_i^j \right)_{j\geq i}$ be a projective sequence of Banach manifolds where the maps $\delta_i^j$ are submersions and  $M= \underleftarrow{\lim}M_i$ 
 its projective limit.\\
A  closed distribution\index{distribution} $\Delta$ on $M$ will be called a
submersive projective limit\index{projective limit!submersive} of local anchored bundles if the following property is satisfied:
\begin{itemize}
\item[(*)]
For any $x= \underleftarrow{\lim}x_i\in M$, there exists an open neighbourhood $\;U=\underleftarrow{\lim}U_i$ of $x$, a submersive  projective sequence of anchored Banach bundles $(E_{i},\pi_{i},U_{i},\rho_{i})$  \footnote{ { \bf  see:} Definition \ref{D_StrongProjectiveLimitOfBanachBundle} and Notations \ref{N_ProjectiveSequenceBundles} for a submersive sequence of projective Banach bundles} \footnote{ { \bf  see:}  Definition \ref{_BanachLieAlgebroid} for an anchored Banach bundle}
\footnote{  \; a sequence of projective anchored bundle $(E_{i},\pi_{i},M_{i},\rho_{i})$ is a projective sequence of Banach bundles which satisfies assumption \textbf{(PSBLA 2)} in Definition \ref{D_ProjectiveSequenceofBanachLieAlgebroids}}
 fulfilling the following properties
   for any $z= \underleftarrow{\lim}z_i \in U$:
 \begin{enumerate}
\item[1.] $ \underleftarrow{\lim}\rho_{i}((E_{i})_{z_i})=\Delta_z$,   for any $z\in U$.  
\item[2.] The kernel of $(\rho_i)_{z_i}$ is complemented  in $(E_i)_{z_i}$ and the range of $(\rho_i)_{z_i}$ is closed, for  all $i\in \mathbb{N}$.
\item[3.] There exists a constant $C>0$ and a  Finsler norm $||\;||^{{E}_i}$ (resp. $||\;||^{{M}_i}$) on $(E_i)_{| U_i}$ (resp. ${TM_i}_{| U_i}$) such that:
		\[
		\forall i \in \mathbb{N}, \; ||(\rho_i)_{z_i}||_i^{\operatorname{op}}\leq C,\; \forall z_i\in U_i.\footnote{\; More precisely,  $||(\rho_i)_{z_i}||_i^{\operatorname{op}}=\sup\{||(\rho_i)_{z_i}(u)||^{M_i},\; ||u||^{E_i}\leq 1\}$}
		\]
	\end{enumerate}
\end{itemize}
\end{definition}
We  have the following criterion of integrability:
\begin{theorem}
\label{T_IntegrabilitySPLBanachBundles}
Let $M$ be a projective limit of a submersive projective sequence  $(M_i, \delta_i^j)_{j\geq i}$ of Banach manifolds and  $\Delta$ be a local  projective
limit of local  Banach bundles on $M$. Assume that, under the property (*), there
exists a Lie bracket $[.,.]_{i}$ on $(E_{i},\pi_{i},U_{i},\rho_{i})$
such that $(E_{i},\pi_{i},U_{i},\rho_{i},[.,.]_{i})$ is a submersive  projective sequence of Banach-Lie algebroids \footnote{\;cf. Definition \ref{D_ProjectiveSequenceofBanachLieAlgebroids}}.\\
Then the distribution $\Delta$ is integrable and the maximal integral manifold $N$ through $x=\underleftarrow{\lim}x_{i}$ is a closed Fr\'echet submanifold of $M$ which is a  submersive projective limit of the set of maximal leaves $N_{i}$ of $\rho_{i}(E_{i})$ through $x_{i}$ in $M_{i}$.
\end{theorem}

{\bf The proof of Theorem takes place in section \ref{ProofIntegrabilitySPLBanachBundles}.  }\\

Now, we  have the following consequences of Theorem \ref{T_IntegrabilitySPLBanachBundles}:

\begin{theorem}\label{T_WeakDistribution}  Let  $\left(  E_{i},\pi_{i},M_{i},\rho_{i}, [.,.]_{i} \right)$ be  a  submersive projective sequence of split  Lie algebroids \footnote{cf. Definition \ref{D_ProjectiveSequenceofBanachLieAlgebroids}}. Then  we have:
\begin{enumerate}
\item[1.]  $\left(E:=\underleftarrow{\lim}  E_{i},\pi:=\underleftarrow{\lim}\pi_{i},M:=\underleftarrow{\lim}M_{i},\rho=\underleftarrow{\lim}\rho_{i} \right)$ is  Fr\'echet anchored bundle and $\Delta=\rho(E)$ is a closed distribution on $M$
\item[2] If  $(\rho_i)$ satisfies the condition (3) in Definition \ref{D_SubmersiveProjectiveLimitOfLocalBanachBundles}, then $\Delta$ is integrable and each leaf $L$ of $\Delta$ is a projective limit of leaves $L_i$ of $\Delta_i$.
\end{enumerate}
\end{theorem}

\begin{proof} The property (1) is a consequence of Proposition \ref{P_ProjectiveLimtAlgebroid}. From this property, it follows that locally $\Delta$ satisfies assumption (1) and (2) of Definition \ref{D_SubmersiveProjectiveLimitOfLocalBanachBundles} so if the assumption (3) is satisfied, the result is a direct consequence of Theorem \ref{T_IntegrabilitySPLBanachBundles}.\\
\end{proof}

\begin{theorem}\label{T_Frobenius} Let $\left( M_i,\delta_i^j \right)_{j\geq i}$ be a submersive projective sequence of Banach manifold and $(E_i, \pi_i, M_i)$ an involutive subbundle of $TM_i$ where $\pi_i$ is the restriction of the natural projection $p_{M_i}:TM_i\to M_i$. Assume that  the restriction  $T\delta_i^j: E_j\to E_i$  is a surjective  bundle morphism for all $i\in \mathbb{N}$ and $j\geq i$. Then $(E_i, \pi_i,M_i)$ is a submersive projective sequence of Banach bundles, and  $(E=\underleftarrow{\lim}E_i, \pi=\underleftarrow{\lim}\pi_i, M=\underleftarrow{\lim}M_i)$ is an integrable Fr\'echet subbundle of $TM$ whose each leaf $L$ of $E$  in $M$ is a projective limit of leaves $L_i$ of $E_i$ in $M_i$
\end{theorem}

\begin{proof} Since $\delta_i^j:M_j\to M_i$ is a surjective submersion, so is $T\delta_i^j:TM_j\to TM_i$.  If $T\delta_i^j:E_j \to E_i$ is a surjective morphism,  this implies that $T\delta_i^j$ is a submersion onto $E_i$ and so $(E_i, \pi_i,M_i)$  is a submersive projective sequence of Banach bundles. Let $\iota_i:E_i\to TM_i$ the natural inclusion and $[\;,\;]_i$ the restriction of Lie bracket of vector fields to  (local) sections of $E_i$. Then $(E_i,M_i, \iota_i, [\;,\;]_i)$ is a Banach Lie algebroid and since $T\delta_i^j\circ \iota_j=\iota_i\circ T\delta_i^j$ it follows that  $(E_i,M_i, \iota_i, [\;,\;]_i)$  is a is a submersive  projective sequence of Banach-Lie algebroids. Fix some $x=\underleftarrow{\lim}x_i\in M=\underleftarrow{\lim}M_i$.
According to Theorem \ref{T_WeakDistribution} we have only to show that the condition (3) of Definition \ref{D_SubmersiveProjectiveLimitOfLocalBanachBundles} is satisfied by $\iota_i$.\\ 
 Given any norm $||\;||^{M_i}$ on $T_{x_i}M_i$ we denote by $||\;||^{E_i}$ the induced norm on the fiber $\{E_i\}_{x_i}$, then, for  the associated   norm operator we have $||\{\iota_i\}_{x_i}||_i^{\operatorname{op}}=1$. So all the assumption of Theorem \ref{T_IntegrabilitySPLBanachBundles}  are satisfied which ends the proof.\\

\end{proof}

\section{Some applications and contre-example}\label{Applications}
\subsection{Application to submersive projective sequence of Banach Lie groups}\label{Application1}

Let  $\left(  G_{i},\delta_{i}^{j} \right) _{j \geq i}  $
be a submersive  a projective sequence of Banach-Lie groups
where $G_{i}$ is modelled on $\mathbb{G}_{i}$ (cf. Definition \ref{D_ProjectiveSequenceofBanachManifolds}). We denote by $\mathbf{L}(G_i)$  the Lie algebra of $G_i$. Then $\mathbf{L}(G_i)\equiv T_eG_i$ is isomorphic to $\mathbb{G}_i$. If we set $\bar{\delta}_i^j:= T_e\delta_i^j$, then each $\bar{\delta}_i^j$ is a surjective linear map from $\mathbf{L}(G_j)$ to $\mathbf{L}(G_i)$  whose kernel is complemented.\\
 Consider a sequence $\mathfrak{h}_i$ of  complemented sub-Lie algebra of  $\mathbf{L}(G_i)$ such that the restriction of $\bar{\delta}_i^j$ to $\mathfrak{h}_j$ is a 
 continuous surjective map. Then $(\mathfrak{h}_i,\hat{\delta}_i^j) _{j \geq i}  $ is a submersive  projective sequence of Banach Lie algebra and so $\mathfrak{h}
 =\underleftarrow{\lim}\mathfrak{h}_i$ is a Fr\'echet Lie algebra (cf. \cite{CaPe} chapter 4). Now from classic result on Banach Lie groups (cf. \cite{Lan}), by left translation each $\mathfrak{h}_i$ gives rise to a complemented involutive subbundle of $\mathcal{H}_i$ of $TG_i$ and the leaf  $H_i$ through the neutral  $e_i$ in $G_i$ has a structure of connected Banach Lie group so that the inclusion $\iota_i: H_i\to G_i$ is a Banach Lie morphism. Note that $\mathfrak{i}_i:=T_{e_i}\iota_i$ is nothing but else that the inclusion 
 of  $\mathfrak{h}_i$  in $\mathbf{L}(G_i)$ and  which induces the natural  inclusion $\hat{\iota}_i$  of $\mathcal{H}_i$ in $TG_i$.\\
 Moreover, since  $(\mathfrak{h}_i,\hat{\delta}_i^j) _{j \geq i}  $ is a submersive  projective sequence of Banach Lie algebra and $\left(  G_{i},\delta_{i}^{j} \right) _{j \geq i}  $
be a submersive  a projective sequence of Banach-Lie groups, it follows that $(\mathcal{H}_i, G_i, \widehat{\mathfrak{i}}_i, [\;,\;]_i)$ \footnote{  here $[\;,\;]_i$ denote again the 
restriction to sections of $\mathcal{H}_i$ of the Lie bracket of vector fields on $G_i$} is a submersive projective sequence of split Banach Lie algebroids. 

On the other hand,  from Theorem \ref{T_PLGLieGroup}   the Lie algebra  $\mathbf{L}(G)$ of $G=\underleftarrow{\lim}{G}_{i}$ is $\underleftarrow{\lim}{\mathbf{L}G}_{i}$ and so
$\mathfrak{h}=\underleftarrow{\lim}\mathfrak{h}_i$ is a closed complemented Lie subalgebra of $\mathbf{L}(G)$. As in the context of Banach setting, by left translation, $
\mathfrak{h}$ gives rise to an involutive Fr\'echet subbundle $\mathcal{H}$ of $TG$ which is clearly the projective limit of the submersive sequence $(\mathcal{H}_i, G_i, 
\widehat{\mathfrak{i}}_i, [\;,\;]_i)$. So from Theorem  \ref{T_Frobenius} by same arguments as in Banach Lie groups, we obtain:

\begin{theorem}
\label{T_IntegrableFrechet Group}
Let ${G}=\underleftarrow{\lim}{G}_{i}$ be a the projective limit of a  submersive   projective sequence of Banach-Lie groups $\left(  G_{i},\delta_{i}^{j} \right) _{j \geq i}  $ and for 
each $i\in \mathbb{N}$, consider  a closed complemented Banach Lie subalgebra  $\mathfrak{h}_i$ of  $\mathbf{L}(G_i)$. Assume that  the restriction of $\bar{\delta}_i^j$ to $\mathfrak{h}
_j$ is a continuous surjective map. Then there exists a Fr\'echet Lie group $H$ in $G$ such that $\mathbf{L}(H)$ is isomorphic to $\mathfrak{h}$ and $H$ is the projective limit 
of $(H_i, {\delta_i^j}_{|H_j})_{j\geq i}$.\\
\end{theorem}

\begin{remark}\label{R_ApplicationComplementaryPropertiesOfProjectiveLimitGroupoids}
The reader will also find an application of Theorem \ref{T_IntegrabilitySPLBanachBundles} in the proof of Theorem 8.23 on submersive projective limit  of a projective sequence of Banach groupoids in \cite{CaPe} which is a kind of generalization of Theorem \ref{T_IntegrableFrechet Group} to Lie groupoids setting .\\
\end{remark}

\subsection{Application to sequences of prolongations of a Banach-Lie algebroid over a Banach manifold}\label{application2}
 Consider an anchored Banach bundle $(\mathcal{A},\pi, M, \rho)$  with typical fiber $\mathbb{A}$. Let $V\mathcal{A}\subset T\mathcal{A}$ be 
 the  vertical subbundle of $p_\mathcal{A}: T\mathcal{A}\to \mathcal{A}$.  If $\mathcal{A}_x:=\pi^{-1}(x)$ is the fiber over $x\in M$, according to \cite{CaPe2}, the prolongation $\mathbf{T}\mathcal{A}$  of the anchored  Banach bundle $( \mathcal{A},\pi,M,\rho)$ over $\mathcal{A}$   is the set $\{(x,a,b,c), (x, b)\in \mathcal{A}_x,\;\; (x,a,c)\in V_{(x,a)}\mathcal{A}\}$. It is a 
 Banach  vector bundle $\hat{\mathbf{p}}: \mathbf{T}\mathcal{A}\to \mathcal{A}$ with typical fiber $\mathbb{A}\times\mathbb{A}$ and we have an anchor $\hat{\rho}: \mathbf{T}\mathcal{A}\to T\mathcal{A}$ given by
 $$\hat{\rho}(x,a,b,c)=(x,a,\rho_x(b), c)\in T_{(x,a)}\mathcal{A}.$$
 
\emph{ From now on, we fix a Banach Lie algebroid $( \mathcal{A},\pi,M,\rho,[.,.]_\mathcal{A})$  such that the typical fiber $\mathbb{A}$ of $\mathcal{A}$ is finite dimensional. By the way, we have a Banach Lie algebroid structure $(\mathbf{T}\mathcal{A},\hat{\mathbf{p}}, \mathcal{A}, \hat{\rho},[.,.]_{\mathbf{T}\mathcal{A}})$} (cf. \cite{CaPe2} Corollary 44  )\\

We denote $(\mathcal{A}_1, \pi_1, \mathcal{A}_0,\rho_1,[.,.]_1)$ the Banach-Lie algebroid $( \mathcal{A},\pi,M,\rho,[.,.]_\mathcal{A})$ over a Banach manifold $\mathcal{A}^0=M$. Thus we have the following commutative diagram:
\begin{eqnarray}
\label{eq_diag_A0TM0}
\xymatrix{
\mathcal{A}_1 \ar[rr]^{\rho_1} \ar[rd]^{\pi_1}& &T\mathcal{A}_0 \ar[ld]^{p_{\mathcal{A}_0}} \\ & \mathcal{A}_0
}
\end{eqnarray}

According to the notations of Theorem 43 and Corollary 44  in \cite{CaPe2},  we set
$\mathcal{A}_2=\mathbf{T}\mathcal{A}_1$,$\;\rho_2=\hat{\rho}$, $\;[.,.]_2=[.,.]_{\mathbf{T}\mathcal{A}_1}$ and $\; \pi_2=\hat{\bf p}$. Then we have the following commutative diagram:

\begin{eqnarray}
\label{eq_diag_A01TM01}
\xymatrix{ \mathcal{A}_2 \ar[rr]^{\rho_2} \ar[rd]_{\pi_2}&& T\mathcal{A}_1 \ar[ld]_{p_{\mathcal{A}_1}}\ar[rd]^{T\pi_1} \\ &\mathcal{A}_1 \ar[rr]^{\rho_1} \ar[rd]_{\pi_1}& &T\mathcal{A}_0 \ar[ld]^{p_{\mathcal{A}_0}} \\ && \mathcal{A}_0}
\end{eqnarray}

Fix some $x\in \mathcal{A}_0$ and a norm $||\;||_0$ (resp.$||\;||_1$) on the fibre $T_x\mathcal{A}_0\equiv\mathbb{A}_0$ (resp. $\mathcal{A}_x=\pi_1^{-1}(x)\equiv\mathbb{A}
_1$). Since  the fiber $\mathbf{T}_{(x,a)}\mathcal{A}_1$ (resp. $T_{(x,a)}\mathcal{A}_1$) is isomorphic to $\mathbb{A}_1\times \mathbb{A}_1$ (resp.  $\mathbb{A}_0\times 
\mathbb{A}_1$), it follows that    $\sup \{||\;||_1\;, ||\;||_1\}$  (resp.  $\sup \{||\;||_0 \;, ||\;||_1\}$ gives rise  to a norm on 
$\mathbf{T}_{(x,a)}\mathcal{A}_1$ (resp. $T_{(x,a)}\mathcal{A}_1$). Then for the associated operator norm $||\;||^{\operatorname{op}}$ we have
\begin{eqnarray}\label{eq_Normrho01}
||(\rho_2)_{(x,a)}||^{\operatorname{op}}\leq \sup(||(\rho_1)_x||^{\operatorname{op}}, 1)
\end{eqnarray}

By induction, for $i\geq 1$, again according to notations of Theorem 43 and Corollary 44 in \cite{CaPe2}  ,  we we set

$\mathcal{A}^{i+1}=\mathbf{T}^{\mathcal{A}_i}\mathcal{A}_i$,$\;\rho_{i+1}=\hat{\rho_i}$, $\;[.,.]_{i+1}=[.,.]_{\mathbf{T}^{\mathcal{A}_i}\mathcal{A}_i}$ and $\; \pi_{i+1}=\hat{\bf p}$
and,  as before, we have the following commutative diagrams:
\begin{eqnarray}
\label{eq_diag_Aii+1TAii+1}
\xymatrix{ \mathcal{A}_{i+1} \ar[rr]^{\rho_{i+1}} \ar[rd]_{\pi_{i+1}}&& T\mathcal{A}_i \ar[ld]_{p_{\mathcal{A}}} \ar[rd]^{T\pi_i} \\ &\mathcal{A}_i \ar[rr]^{\rho_i} \ar[rd]_{\pi_i}& &T\mathcal{A}_{i-1} \ar[ld]^{p_{\mathcal{A}_{i-1}}}  \\ && \mathcal{A}_{i-1}}
\end{eqnarray}

Also, by same arguments as for (\ref{eq_Normrho01}) we obtain:
\begin{eqnarray}
\label{eq_Normrhoii+1}
||(\rho_{i+1})_{(x,a_1,\dots,a_{i+1})}||^{\operatorname{op}}\leq \sup(||(\rho_i)_{(x,a_1,\dots,a_i)}||^{\operatorname{op}}, 1)
\end{eqnarray}

It follows that we have a submersive projective sequence of Banach-Lie algebroids $\left( \mathcal{A}_i, \mathcal{A}_{i-1}, \rho_i, [.,.]_i \right) $ over a submersive  projective sequence of Banach manifolds $ \left( \mathcal{A}_{i} \right) _{i \in \mathbb{N}}$ which satisfies the assumptions of Theorem \ref{T_IntegrabilitySPLBanachBundles}. Thus,  we obtain:
\begin{theorem}
\label{T_SequenceProlongation}
Under the previous context,
\[
(\mathcal{A}= \underleftarrow{\lim}_{i\geq1}\mathcal{A}_i, \mathcal{M}=\underleftarrow{\lim}_{j\geq 0}\mathcal{A}_j, \rho=\underleftarrow{\lim}_{i\geq 1}\rho_i, [.,.]=\underleftarrow{\lim}_{i\geq 1}[.,.]_i)
\]
is a Fr\'echet Lie algebroid on the Fr\'echet manifold  $\mathcal{M}$ and the distribution $\rho(\mathcal{A})$ is integrable.  Each leaf ${L}$ is a projective  limit of a projective sequence of leaves of type $(L_i)$ defined by induction in the following way:

$L_0$ is a leaf of $\rho_1(\mathcal{A}_1)$ and if $L_i$ is a leaf of $\rho_i(\mathbb{A}_i)$ then $L_{i+1}=(\mathcal{A}_i)_{| L_i}$.\\
\end{theorem}

\subsection{ A contre-example}\label{Contre-example} In this subsection we give a Example  of an integrable distribution on a Fr\'echet bundle over a finite dimensional manifold which    satisfies the assumptions (*) in  Definition \ref{D_SubmersiveProjectiveLimitOfLocalBanachBundles} but is  a projective limit of a submersive sequence of Banach {\bf not integrable} distributions.\\

Let $E=M\times\mathbb{R}^m$ the trivial bundle over a manifold $M$ of dimension $n$. The set $J^k(E)$ of the $k$-jets of section of $E$ over $M$ is a finite dimensional manifold  which is a vector bundle $\pi^k:J^k(E)\to M$ and whose typical  fiber is  is the space
$
{\displaystyle\prod\limits_{j=0}^{k}}
\mathcal{L}_{\mathrm{sym}}^{j}\left(  \mathbb{R}^{n},\mathbb{R}^{m}\right)
$
where $\mathcal{L}_{\mathrm{sym}}^{j}\left(  \mathbb{R}^{n},\mathbb{R}%
^{m}\right)  $ is the space of continuous $j$-linear symmetric mappings
$\mathbb{R}^{n}$ $\to$ $\mathbb{R}^{m}.$ Then each projection
$\pi_{k}^{l}:J^{l}\left(  E\right)  \to J^{k}\left( E\right)  $
defined, for $l\geq k$, by
\[
\pi_{k}^{l}\left[  j^{l}\left(  s\right)  \left(  x\right)  \right]
=j^{k}\left(  s\right)  \left(  x\right)
\]
is a smooth surjection.

\begin{proposition}
\label{P_FrechetStructureOnProjectiveLimitOfJetsOfBanachVectorBundles} (\cite{CaPe})
$\left(  J^{k}\left(  E\right)  ,\pi_{k}^{l}\right)$ is a  submersive projective sequence of Banach
vector bundles and the projective limit
\[
J^{\infty}\left(  E\right)  =\underleftarrow{\lim}J^{k}\left(  E\right)
\]
can be endowed with a structure of Fr\'{e}chet vector bundle whose fibre is isomorphic to the Fr\'{e}chet space
${\displaystyle\prod\limits_{j=0}^{\infty}}
\mathcal{L}_{\mathrm{sym}}^{j}\left(  \mathbb{R}^{n},\mathbb{R}^{m}\right)$.\\
\end{proposition}

Let $s$ be a section of $\pi$ on a neighbourhood $U$ of $x\in M$. For
$\xi=j^{k}\left(  s\right)  \left(  x\right)  \in J^{k}\left( E\right)  $, the $n$-dimensional subspace $R\left(  s,x\right)  $ of $T_{\xi}J^{k}\left(E\right)  $ equals to the tangent space at $\xi$ to the submanifold  $j^{k}\left(  s\right)  \left(  U\right)  \subset
J^{k}\left(  \pi\right)  $ is called an  \textit{$R$-plane}\index{$R$-plane}.

The Cartan subspace $\mathcal{C}^k\left(  E \right)  $ of $T_{\xi}J^{k}\left(
E \right)  $ is the linear subspace spanned by all $R$-planes
$R\left(  s^{\prime},x\right)  $ such that $j^{k}\left(  s^{\prime}\right)
\left(  x\right)  = \xi$. So it is the hull of the union of
$\left(  j^{k}\left(  s\right)  \right)  _{\ast}\left(  x\right)  \left(
T_{x}M\right)  $ where $s$ is any local section of
$\pi$ around de $x$.

 The Cartan subspaces form a smooth distribution on $J^{k}\left(  \pi\right)  $ called
\textit{Cartan distribution}and denoted $\mathcal{C}^{k}$. Then $\mathcal{C}^{k}$ is a regular distribution which  a contact distribution  and so is not integrable (cf. \cite{Igo}). We have a submersive projective limit of  bundle $(\mathcal{C}^{k}, T\pi^l_k, j^k(E))$ whose projective limit $\mathcal{C}= \underleftarrow{\lim}\mathcal{C}^{k}$ is called the Cartan distribution on $J^{\infty}\left(  E\right)$.  In fact $\mathcal{C}$ is integrable (cf. \cite{Igo}). Note that since $\mathcal{C}^{k}$ is a subbundle of $TJ^k(E)$, from the proof of Theorem \ref{T_Frobenius}, the condition (3) of Definition \ref{D_SubmersiveProjectiveLimitOfLocalBanachBundles} is satisfied.\\

\section{Proof of Theorem \ref{T_IntegrabilitySPLBanachBundles}}\label{ProofIntegrabilitySPLBanachBundles} 

Fix some  $x\in M$. According to the property (*) and the Definition \ref{D_SubmersiveProjectiveLimitOfLocalBanachBundles},  we can choose a  submersive projective sequence of charts  $\left(U_i,{\delta_i^j}_{| U_j}\right)_{j\geq i}$ and a submersive projective sequence of Banach bundles $(E_i, \lambda_i^j)_{j\geq i}$, such that:
	\begin{itemize}
		\item[--]
		$U=\underleftarrow{\lim}U_i$ is an open neighbourhood of $x$ in $M$;
		\item[--]
		If $x=\underleftarrow{\lim}(x_i)$, each $U_i$ is the contractile domain of a  chart $(U_i,\phi_i)$ around $x_i$ in $M_i$ and $(U=\underleftarrow{\lim}(U_i),\phi=\underleftarrow{\lim}(\phi_i))$ is a projective limit chart in $M$ around $x$.
	\item[--] the projective sequence of Banach bundles $(E_i, \lambda_i^j)_{j\geq i}$ satisfies the assumption (*) on $U$.\\
	\end{itemize}
	
\textit{Step 1: The kernel of $\rho_x$ is supplemented.}\\
	There exists a trivialization $\tau_i: E_i\to  U_i\times \mathbb{E}_i$ which satisfies the compatibility condition:
	\begin{eqnarray}
	\label{eq_CompatibleFdt}
	(\delta_i^j\times\overline{\lambda_i^j})\circ\tau_j=\tau_i\circ \lambda_i^j
	\end{eqnarray}
	where $(\mathbb{E}_i,\overline{\lambda_i^j})_{j\geq i}$ is the projective sequence of Banach spaces  on which $(E_i, \lambda_i^j)_{j\geq i}$ is modeled.\\
	Under these conditions, without loss of generality we may assume that, for each $i\in \mathbb{N}$, we have
	\begin{itemize}
		\item[--] $x_i\equiv 0\in \mathbb{M}_i$;
		\item[--] $U_i$ is an open subset of $\mathbb{M}_i$ and so  $TU_i=U_i\times\mathbb{M}_i$;
		\item[--] $E_i=U_i\times \mathbb{E}_i$.		
	\end{itemize}
The projection $\delta_i^j$  at point $x_j\equiv 0$ (resp. $\lambda_i^j$ in restriction to the fibre of $E_j$ over $x_j\equiv 0$) is denoted $d_i^j$ (resp. $\ell_i^j$). The morphism $\rho_i$ in restriction to the fibre $E_i$ over $x_i\equiv 0$ is denoted $r_i$ and  $\rho_x$ is denoted $r$, so that  $r=\underleftarrow{\lim}r_i$. Now, according to the context of  Assumption 1 in Definition \ref{D_SubmersiveProjectiveLimitOfLocalBanachBundles}, we have the following result:
\begin{lemma}
	\label{L_kerrhox}
	There exists a decomposition
	$\mathbb{E}=\ker r\oplus \mathbb{F}'$  with the following property:\\
	 if $(\nu'_n)$ (resp.  $(\mu_n)$ is  the graduation  on $\mathbb{F}'$ (resp.  $(\mu_n)$ on  $\mathbb{M}$)
	  induced by the norm $||\;||^{\mathbb{E}_i}$ (resp. $||\;||^{\mathbb{M}_i}$) on $(E_i)_{x_i}$ (resp. $T_{x_i}M_i$),  then the restriction of $r$ to $\mathbb{F}'$ is a closed uniformly  bounded operator according to these graduations \footnote{cf. Appendix \ref{__TheFrechetSpaceHF1F2-TheBanachSpaceHbF1F2}}.
	\bigskip
\end{lemma}
\begin{proof}{\it of Lemma \ref{L_kerrhox}}	${}$\\

At first, in such a context we have $T_{x_j}\delta_i^j\equiv d_i^j$ and the following compatibility condition:
\begin{eqnarray}
\label{eq_rCommutativeDiagram}
		 d_i^j\circ r_j=r_i\circ \ell_i^j.
\end{eqnarray}
		
We set $\mathbb{F}_i=r_i(\mathbb{E}_i)$ for all $i\in \mathbb{N}$.   From Definition \ref{D_SubmersiveProjectiveLimitOfLocalBanachBundles} assumption 1,  $\mathbb{F}_i$ is a Banach subspace of $\mathbb{E}_i$ and  there exists a decomposition $\mathbb{E}_i=\ker r_i\oplus \mathbb{F}'_i$. Thus  the restriction $r'_i$ of $r_i$ to $\mathbb{F}'_i$ is an isomorphism onto $\mathbb{F}_i$.
Now, from (\ref{eq_rCommutativeDiagram}), we have
\[
		\forall (i,j)\in \mathbb{N}^2:j>i, \; d_i^j\circ r'_j=r'_i\circ \ell_i^j.
\]
But since each $r'_i$ is an isomorphism, the restriction $(\ell_i^j)'$  of $\ell_i^j$ to $\mathbb{F}'_j$ takes values in $\mathbb{F}'_i$ for all $\left( i,j \right)  \in \mathbb{N}^2$ such that $j\geq i$. Moreover, as $\delta_i^j$ is surjective, according to (\ref{eq_rCommutativeDiagram}) again, this implies that $d_i^j(\mathbb{F}_j)=\mathbb{F}_i$ and so $\ell_i^j(\mathbb{F}'_j)=\mathbb{F}'_i$. Since $(\mathbb{E}_i,\ell_i^j)_{j\geq i}$ is a projective sequence,  this implies that  $(\mathbb{F}'_i,(\ell_i^j)')_{j\geq i}$ is a surjective projective system. The vector space $\mathbb{F}'=\underleftarrow{\lim}\mathbb{F}'_i$ is then a Fr\'echet subspace of $\mathbb{E}$. \\
On the other hand, let  $(\ell_i^j)^{''}$ be the restriction of $\ell_i^j$ to $\ker r_j$. Always from  (\ref{eq_rCommutativeDiagram}), we have
		\[
		\forall (i,j)\in \mathbb{N}^2:j>i, \;(\ell_i^j)^{''}(\ker r_j)\subset \ker r_i.
		\]
By same argument, it implies  that $(\ker r_i, (\ell_i^j)^{''})_{j\geq i}$ is a projective sequence and $\ker  r= \underleftarrow{\lim}\ker r_i$. Moreover, since $\mathbb{E}_i=\ker r_i\oplus \mathbb{F}'_i$, it follows that $\mathbb{E}=\ker r\oplus \mathbb{F}'$ and also the restriction $r'$ of $r$ to $\mathbb{F}'$  is obtained as $r'= \underleftarrow{\lim}r'_i$ and $r'$ is an injective continuous operator $r':\mathbb{F}'\to \mathbb{M}$ whose range is the closed subspace $\mathbb{F}=\underleftarrow{\lim}\mathbb{F}_i$. It remains to show that $r'$ is uniformly bounded.\\
According to the context of assumption 2 of Definition \ref{D_SubmersiveProjectiveLimitOfLocalBanachBundles}, there exists a constant $C>0$  and, for each $i\in \mathbb{N}$,  we have a norm $\|\;\|^{\mathbb{E}_i}$ on $\mathbb{E}_i$ and a norm $\|\;\|^{\mathbb{M}_i}$ on $\mathbb{M}_i$  such that $\|r_i\|^{\operatorname{op}}_{i}\leq C$. As $\mathbb{F}'_i$  is a closed Banach subspace of $\mathbb{E}_i$, it follows that for the induced norm on $\mathbb{F}'_i$ we have 		
\begin{eqnarray}
\label{eq_r'iC}
		\forall i\in \mathbb{N},\;||r'_i||^{\operatorname{op}}_{i}\leq C.
\end{eqnarray}
Set $\ell_i=\underleftarrow{\lim}\ell_i^j$ and $d_i=\underleftarrow{\lim}d_i^j$ . By construction we have $\ell_i(\mathbb{F}')=\mathbb{F}'_i$ and $d_i(\mathbb{M})=\mathbb{M}_i$. The norm $\|\;\|^{\mathbb{E}_i}$ on $\mathbb{E}_i$  induces a norm $\|\;\|^{\mathbb{F}'_i}$ on the Banach subspace $\mathbb{F}'_i$ and  we get a natural graduation $(\nu'_i)$ on $\mathbb{F'}$  given by $\nu'_i(u)=||\ell_i(u)||^{\mathbb{F}'_i}$  (cf.  Appendix 
\ref{__TheFrechetSpaceHF1F2-TheBanachSpaceHbF1F2} (\ref{eq_SemiNormOnFrechetAssociatedToNormOnBanach})). In the same way, the norm  $\|\;\|^{\mathbb{M}_i}$ 
induces a graduation $(\mu_i)$ on $\mathbb{M}$ given by $\mu_i(v)=||d_i(v)||^{\mathbb{M}_i}$. Now (\ref{eq_r'iC}) implies that $r'$ is uniformly bounded and $r'(\mathbb{F}')=r(\mathbb{E})=\Delta_x$ is closed by 
assumption. Therefore, the proof of Lemma \ref{L_kerrhox} is complete.\\		 
\end{proof}

\textit{Step 2:  There exists a neighbourhood  $V\subset U$ of $x$ such that the map $\rho'=\rho_{| U\times \mathbb{F}'}$ takes values in $\mathcal{IH}_b(\mathbb{F}',\mathbb{M})  $\footnote{ cf. Appendix \ref{__TheFrechetSpaceHF1F2-TheBanachSpaceHbF1F2}} and is $K$-Lipschitz on $V$ for some $K>0$}.\\
	
Since $E_i=U_i\times \mathbb{E}_i$, it follows that  $E=U\times \mathbb{E}$ and so $\rho:E\to TU$ can be seen as a smooth map from $U$ into $\mathcal{H}(\mathbb{E},\mathbb{M})$. Let $\rho'$ be the restriction of $\rho$ to $U\times \mathbb{F}'$ and so consider $\rho'$ as a smooth map from $U$ to  $\mathcal{H}(\mathbb{F}',\mathbb{M})$.  From  the definition of a Finsler norm, Assumption 2 in Definition \ref{D_SubmersiveProjectiveLimitOfLocalBanachBundles} and    Lemma \ref{L_kerrhox},  the map $x\mapsto\rho'_x$  takes value in $\mathcal{H}_b(\mathbb{F}',\mathbb{M})$. 

\begin{lemma}
	\label{L_rho'lipschitz}
	There exists a neighbourhood  $V_1\subset V$ of $0$  such that the map $\rho':V_1\to \mathcal{H}_b(\mathbb{F}',\mathbb{M})$  is  Lipschitz, that is:
	
	there exists $K>0$ such that	
		$$\hat{\mu}_i^{op}(\rho'_z-\rho'_x)\leq K \hat{\mu}_i(z-z'),\; \forall (z,z')\in V_1^2$$\footnote{ cf. Remark \ref{R_SeminormEquivalentpn}}
\end{lemma}
	
\proof{\it  of Lemma \ref{L_rho'lipschitz}}
Let $( \hat{\nu'}_i)_{i \in \mathbb{N}} $ (resp. $\left( \hat{\mu}_i\right)_{i \in \mathbb{N}} $) be the canonical increasing graduation associated to the graduation $(\nu'_i)_{i \in \mathbb{N}}$ on $\mathbb{M}$ (resp. $\left( \hat{\mu}_i\right)_{i \in \mathbb{N}} $) (cf. Appendix (\ref{__TheFrechetSpaceHF1F2-TheBanachSpaceHbF1F2} \ref{eq_SemiNormOnFrechetAssociatedToNormOnBanach})).
Since the map $x\mapsto \rho_x'$ is a smooth map from $U$ to $\mathcal{H}_b(\mathbb{F}',\mathbb{M})$ it follows that for each $x$ the differential map $d_x\rho'$ is a continuous linear map from $\mathbb{M}$ to the Banach space $\mathcal{H}_b(\mathbb{F}',\mathbb{M})$ and so there exists $i_0\in \mathbb{N}$  and a constant $A_x>0$ such that 
\begin{equation}\label{boundrho1}
\|d_x\rho'(u)\|_\infty\leq  A_x \;\hat{\nu'}_{i_0}(u)
\end{equation}
for all $u\in \mathbb{F}'$ and so 
\begin{equation}\label{boundrho2}
\|d_x\rho'(u)\|_\infty\leq A_x\;\hat{\nu'}_{i}(u)
\end{equation}
for all $u\in \mathbb{F}'$ and  and $i\geq i_0$ and  according to Remark \ref{R_SeminormEquivalentpn} we set
\begin{equation}\label{normrhoop}
||d_x\rho'||_i^{op}:=\sup\{\hat{\mu_i}(d_x \rho'(u)):\; \hat{\nu'}_i(u)\leq 1\}\leq A_x,\;\forall i\geq i_0
\end{equation}
On the other hand, for $1\leq i<i_0$ we set $C_x^i=||d_x\rho'||_i^{op}$.  Then   $d_x\rho'$ belongs to $\mathcal{H}_b\left(\mathbb{M},\mathcal{H}_b(\mathbb{F}',\mathbb{M})\right)$ for all $x\in U$ since  we have 
$$||d_x\rho'||_\infty:= \sup_{i\in \mathbb{N}}||d_x\rho'||_i^{op}\leq\sup\{A_x, C_x^1,\cdots, C_x^{i_0-1}\}.$$
We set  $C=||d_0\rho'||_\infty$. By continuity, there exists an open neighbourhood $V_1$ of $0$ such that
\begin{equation}\label{norminfinirho}
 ||d_x\rho'||_\infty\leq 2C
\end{equation}
 By choosing $K=2C$,  from the definition of $||d_x\rho'||_\infty$  it implies the announced result.

\endproof

\textit{Step 3: Local flow   of the vector field $X_u=\rho'(u)$}.\\
Consider a neighbourhood $V$ as announced  in step 2. As $\delta_i$ is surjective, $V_i=\delta_i(V)$ is an open set of $U_i$ and so we have $V=\underleftarrow{\lim}V_i$. For each $u\in \mathbb{F}'$, let $X_u=\rho'(u)$ be the vector field on $V$. If $u=\underleftarrow{\lim}u_i$ with $u_i\in \mathbb{F}'_i$, then $X_{u_i}=\rho'_i(u_i)$ is a vector field on $V_i$ and $X_u=\underleftarrow{\lim} X_{u_i}$. Now since $\rho'$ takes values in $\mathcal{IH}_b(\mathbb{F}',\mathbb{M})$,  from our assumption,  there exists a constant $C>0$ such that  $||(\rho'_i)_{z_i}||^{\operatorname{op}}_i\leq C$ for all $z_i\in V_i$. Therefore, if $u$ belongs to $\mathbb{F}'$  and $u_i=\lambda(u)$, then $\|X_{u_i}\|^{\mathbb{M}_i}\leq C \|u\|^\mathbb{E}_i$ and, from Lemma \ref{L_rho'lipschitz} and the definition of $\hat{\mu}_i^{\operatorname{op}}$  we have
\begin{eqnarray}
\label{eq_LipschitzXui}
	\forall \left( x_i,\; x'_i\right) \in \left( \delta_i(V)\right) ^2, \; \|X_{u_i}(x_i)-X_{u_i}(x'_i)\|^{\mathbb{M}_i} \leq K \|u_i\|^{\mathbb{E}_i}||x_i- x'_i||^{\mathbb{M}_i}.
\end{eqnarray}
Now,  recall that we have  provided $\mathbb{F}'$ and  $\mathbb{M}$  with  seminorms $(\hat{\nu'}_n)$ and $(\hat{\mu}_n)$ respectively defined by
	\[
	{\nu}_n(u)= 
	 \|\lambda_i(u)\|^{\mathbb{E}_i} \; \textrm{ and }{\mu}_n (x)=
	  \|\delta_i(x)\|^{\mathbb{M}_i}.
	\]	
Since $\delta_i(X_{u})(x)=X_{u_i}(\delta_i(x))$ and in this way, from (\ref{eq_LipschitzXui}), for all $n\in \mathbb{N}$,  we have
\begin{eqnarray}
\label{eq_LipschitzXu}
	 \forall (x,x')\in V^2, \; {\mu}_n(X_{u}(x)-X_{u}(x'))\leq K {\nu}_n(u){\mu}_n(x-x').
\end{eqnarray}
Therefore, $X_u$ satisfies the assumption of Corollary \ref{C_ODEVectorFields}. Let $\epsilon>0$ such that the pseudo-ball
	\[
B_\mathbb{M}(0,2\epsilon)=	\{x\in \mathbb{M},\;:\; \hat{\mu}_{n_i}(x)<2\epsilon, 1\leq i\leq k\}
	\]
is contained in $V$ and set
\begin{itemize}
\item[]
$C_1:=\displaystyle\max_{1\leq i\leq k} \{K\nu_{n_i}(u)\}=K \displaystyle\max_{1\leq i\leq k}{\nu}_{n_i}(u)$;
\item[]
$C_2:=\displaystyle\sup_{x \in B_{\mathbb{M}}(0,\epsilon)}\left\{\max_{1\leq i\leq k}\mu_{n_i}(X(x))\right\}\leq C \displaystyle\max_{1\leq i\leq k}{\nu}_{n_i}(u) $.
\end{itemize}	
By application of Corollary \ref{C_ODEVectorFields}, for any $u$ such that $ \displaystyle\max_{1\leq i\leq k}{\nu}_{n_i}(u)\leq 1$, there exists  $\alpha>0$ such that
	\[
	\alpha e^{2\alpha K}\leq \frac{\epsilon}{2C},
	\]
 such that the local flow $\operatorname{Fl}^{u}_t$ is defined on the pseudo-ball $B_\mathbb{M}(0,\epsilon)$ for all $t\in [-\alpha,\alpha]$ for all $u$ which satisfied the previous inequality

Note that for any $s\in \mathbb{R}$ we have $X_{su}=sX_u$.  Therefore, from the classical properties of a flow of a vector field,  

 there exists $\eta>0$ such that  the local flow  $\operatorname{Fl}^{u}_t$  is defined on $[-1,1]$, for all $u $ in  the open pseudo-ball
	\[
	B_{\mathbb{F}'}(0,\eta):=\left\{u\in \mathbb{F},\;:\;\ \; {\nu}'_{n_i}(u)\leq \eta,\;\; 1\leq  i\leq k\right\}.
	\]
	(cf. for instance proof of Corollary 4.2 in \cite{ChSt}).\\
We set $B_{\mathbb{M}_i}(0,\epsilon)=\delta_i(B_\mathbb{M}(0,\epsilon))$. Then $X_{u_i}$ is a vector field on $V_i=\delta_i(V)$ and
 $\operatorname{Fl}^{u_i}_t:=\delta_i\circ(\operatorname{Fl}^{u}_t)\circ \delta_i$  is the local flow of $X_{u_i}$ which is defined on  $B_{\mathbb{M}_i}(0,\epsilon)$ for all $t \in [-1,1]$ and
 $\operatorname{Fl}^{u_i}_t (x_i)$ belongs to $V_i$  for all $x_i\in B_{\mathbb{M}_i}(0,\epsilon)$ and $t\in [-1,1]$ and, from (\ref{projlimflt}), we have
\begin{equation}\label{projlimflt}
\operatorname{Fl}^{u}_t=\underleftarrow{\lim}\operatorname{Fl}^{u_i}_t.\\
\end{equation}
\newline
\bigskip	
\textit{Step 4: Existence of an integral manifold}.\\
Since $(E_{i},\pi_{i},U_{i},\rho_{i},[.,.]_{i})$ is a  Banach-Lie algebroid, the Lie bracket $[X_{u_i}, X_{u'_i}]$ is tangent to $\Delta_i$ and so by Definition 3.2 and  Lemma 3.6 in  \cite{Pe}  we have
	\[
	\forall t\in [-1,1], \; (T\operatorname{Fl}^{u_i}_t)((\Delta_i)_{x_i})=(\Delta_i)_{\operatorname{Fl}^{u_i}_t(x_i)}.
	\]
Therefore, according to the notations at the end of step 3, for each $i\in \mathbb{N}$, we set:
 \[
 \forall u_i\in B_{\mathbb{F}'_i}(0,\eta)=\lambda_i(B_{\mathbb{F}'}(0,\eta)), \; \Phi_i(u_i)=(\operatorname{Fl}^{u_i}_1)(0);
 \]
 \[
 \forall u\in B_{\mathbb{F}'}(0,\eta), \; \Phi(u)=(\operatorname{Fl}^{u}_1)(0);.
 \]
	
\begin{lemma}
\label{L_closedFrechet}
$\Phi=\underleftarrow{\lim}\Phi_i$ is smooth and there exists $0<\eta'\leq \eta$  such that the restriction of $\Phi$ to $B_{\mathbb{F}'}(0,\eta')$ is injective and $T_u\Phi$ belongs to $\mathcal{IH}_b(\mathbb{F}',\mathbb{M})$ for all $u\in B_{\mathbb{F}'}(0,\eta')$.
\end{lemma}

Recall that, for each $i\in \mathbb{N}$,  $B_{\mathbb{F}'_i}(0,\eta'))$ is an open ball in $\mathbb{F}'_i$ and  $\delta_i(V)$ is  open neighbourhood of $0\in \mathbb{M}_i$.   From Lemma \ref{L_closedFrechet},  since $\Phi$ is injective and each differential $T_u\Phi$ is injective, it follows that the same is true for each $\Phi_i: B_{\mathbb{F}_i}(0,\eta'))\to \delta_i(V)$. Thus we can apply the proof of Theorem 1 in \cite{Pe}  for $\Phi_i$. By the way,  $\mathcal{U}_i=\Phi_i(B_{\mathbb{F}'_i}(0,\eta'))$ is an integral manifold of $\rho_i(\mathbb{F}'_i)$  and $(\mathcal{U}_i, \Phi_i^{-1})$ is a (global) chart for this integral manifold modeled on $\mathbb{F}'_i$. As $ \Phi=\underleftarrow{\lim}\Phi_i$, $B_{\mathbb{F}'}(0,\eta')=\underleftarrow{\lim}B_{\mathbb{F}'_i}(0,\eta'))$ $\Phi_i\circ \lambda_i^j=\delta_i^j\circ \Phi_j$ for all $j\geq i$, it follows that $(\mathcal{U}_i, \delta_i^j)_{j\geq i}$ is a surjective  projective sequence and so $\mathcal{U}=\underleftarrow{\lim}\mathcal{U}_i$ is a Fr\'echet manifold modeled on $\mathbb{F}'$. This last result clearly {\bf ends the proof of Theorem \ref{T_IntegrabilitySPLBanachBundles}}.	\\

\bigskip
	
\begin{proof}{\it of Lemma \ref{L_closedFrechet}}${}$\\
According to step 3, $\Phi_i$ is well defined on $B_{\mathbb{F}'_i}(0,\eta)$ and
$\Phi=\underleftarrow{\lim}\Phi_i$.\\
Now, for every $ x\in V$,  $\rho'_x$  belongs to $\mathcal{IH}_b(\mathbb{F}',\mathbb{M})$ and so is injective; after shrinking $V$, if necessary, there exists some $M>0$ such that 		
\begin{eqnarray}
\label{eq_boundrho'}
\forall x\in V, \; \|\rho'_x\|_\infty\leq M.
\end{eqnarray}
Now, by construction, we have  $T_0\Phi(u)=\rho'_0(u)$ and so $T_0\Phi$  is injective.
		
\begin{claim}
\label{Cl_BoundDPhi}
The map $u\mapsto T_u\Phi$ is a smooth map from $B'(0,\eta )$ to $\mathcal{H}_b(\mathbb{F}',\mathbb{M})$.\\
\end{claim}

\smallskip 
\begin{proof}{\it  of Claim \ref{Cl_BoundDPhi}}${}$\\
We will use some argument of the proof of Lemma 2.12 of \cite{Pe}. We fix the index $i$ and, for any $y\in B_{\mathbb{F}'_i}(0,\epsilon )$, $v\in \mathbb{F}'_i$ we set
\begin{description}
\item[ ] $X_v(y)=\rho'_i(y, v)$
\item[ ] $\varphi(t,v) =\operatorname{Fl}^{X_v}_t(0)$
\item[ ] $A(t)= \partial_1 \rho'_i(\varphi(t,v),v)$ (partial derivative relative to the first variable)
\item[ ] $B(t)= \rho'_i(\varphi(t,v),.)$
\end{description}
Note that $A$ and $B$ are smooth fields on $[0,1]$ of operators in $\mathcal{L}(\mathbb{M}_i,\mathbb{M}_i)$ and $\mathcal{L}(\mathbb{F}'_i,\mathbb{M}_i)$ respectively. Therefore the differential equation
\[
\dot{S}= A\circ S+B
\]
has a unique solution $S_t$ with initial condition $S_0=0$ given by
\begin{eqnarray}
\label{eq_St}
S_t=\displaystyle\int_0^t G_{t-s}\circ B(s)ds
\end{eqnarray}
where $G_t$ is the unique solution of
\[
\dot{G}=A\circ G
\]
with initial condition $G_0=\operatorname{Id}_\mathbb{M}$. Given by
\begin{eqnarray}
\label{eq_Gt}
G_t=\operatorname{Id}_\mathbb{M}+\displaystyle\int_0^t A\circ G_s ds.
\end{eqnarray}
Under these notations, from  \cite{Die}, Chapter X $\S$ 7, we have
$\partial_2\varphi(t,v)(.)=S_t$ .\\

On the one hand, from the choice of $\epsilon$,  $\varphi(t,v)$ belongs to $V_i$ and so from (\ref{norminfinirho}), we have $\|A(t)\|^{\operatorname{op}}_i\leq K$ for any $t\in [-1,1]$ and from (\ref{eq_boundrho'}) we have $\|B(t)\|^{\operatorname{op}}_i\leq M$.
Thus from (\ref{eq_Gt}), using Gronwal equality,
we obtain
\[
\|G_t\|^{\operatorname{op}}_i\leq e^K,
\]
And so from (\ref{eq_St}) we obtain
\[
\|S_t\|^{\operatorname{op}}_i\leq Me^K
\]
This implies that
\[
\|\partial_2\varphi(t,v)\|^{\operatorname{op}}_i\leq Me^K.
\]
We set $M_1=Me^K$. Since $\Phi_i(u_i)=\varphi(1,u_i)$ it follows that:
\begin{eqnarray}\label{eq_TuiPhi}
\|T_{u_i}\Phi(v_i)\|^{\mathbb{M}_i}\leq M_1\|v_i\|^{\mathbb{E}_i}\end{eqnarray}

 from (\ref{eq_TuiPhi}) we obtain:

\begin{eqnarray}
\label{eq_muiTPhiu}
{\mu}_i(T_{u}\Phi(v))\leq M_1{\nu}_i(v).
\end{eqnarray}
But, since $T_u\Phi(\mathbb{F}')=\Delta_{\Phi(u)}\subset \{\Phi(u)\}\times \mathbb{M}$, it follows that the map $u\mapsto T_u\Phi$ can be considered as a continuous linear map from $\mathbb{F}'$ to $\mathbb{M}$ which takes values in $\mathcal{H}_b(\mathbb{F}',\mathbb{M})$. Now as  each $\Phi_i$ is a smooth map form $ B_{\mathbb{F}_i}(0,\eta'))$ to $\delta_i(V)$  and     $\Phi=\underleftarrow{\lim}\Phi_i$, this imply that $\Phi$ is a smooth map on $B_{\mathbb{F}'}(0,\eta'))=\underleftarrow{\lim} B_{\mathbb{F}'_i}(0,\eta'))$ to $V=\underleftarrow{\lim}\delta_i(V)$  which ends the proof of the Claim.

\end{proof}

\emph{End of the proof of Lemma \ref{L_rho'lipschitz}}.\\
At first, from Claim  \ref{Cl_BoundDPhi},  the map $u\mapsto T_{u}\Phi$ takes values in  the Banach space $\mathcal{H}_b(\mathbb{F}',\mathbb{M})$, as in step 2 for $\rho$, we can show that this map is Lipschitz  on $B_{\mathbb{F}'}(0,\eta)$  for $\eta$ small enough.
As $T_0\Phi=\rho'_0$, from Proposition \ref{P_InjectiveSurjectiveBH}, it follows that, again for $\eta$ small enough,  $T\Phi$ is injective on $B_{\mathbb{F}}(0,\eta')$, and we have (cf. (\ref{eq_TuiPhi})	)
\begin{eqnarray}
\label{eq_muiTPhiuMK}
\forall  u \in B_{\mathbb{F}}(0,\eta'), \; \forall v\in \mathbb{F}' \; {\mu}_i(T_{u}\Phi(v))\leq M_1 {\nu}_i(v)
\end{eqnarray}	
using the fact that the range of $T_u\Phi$ is always closed for $u\in B_{\mathbb{F}'}(0,\eta')$.
Moreover,  for $u\in B_{\mathbb{F}'}(0,\eta')$,  as for the relation (\ref{eq_Lisobounded}) in the proof of Proposition \ref{P_InjectiveSurjectiveBH}, we obtain:
\begin{eqnarray}
\label{eq_TPhiBounded}
		\frac{1}{\ell_u}{\nu}_i(v)\leq\mu_i(T_u\Phi(v)))\leq \ell_u.{\nu}_n(v)
\end{eqnarray}
for all $i\in \mathbb{N}$, where $\ell_u=\|T_u\Phi \|_\infty\leq M_1$.\\
Finally we obtain:
\begin{eqnarray}
\label{eq_TPhiGlobalBounded}
		\forall i \in \mathbb{N}, \; \frac{1}{M_1}{\nu}_i(v)\leq {\mu}_i(T_u\Phi(v))\leq M_1.{\nu}_i(u)
\end{eqnarray}
Suppose that, for any $0\leq \eta'\leq\eta$, the restriction of each $\Phi_i$  to $B_{\mathbb{F}'_i}(0,\eta)$ is not injective. Consider any pair $(u,v)\in[ B_{\mathbb{F}'}(0,\eta)]^2$ such that $u\not =v$ but $\Phi(u)=\Phi(v)$,
we set $h=v-u$.  For any $\alpha \in \mathbb{M}^*$, we consider the smooth curve $c_\alpha:[0,1]\to \mathbb{R}$ defined by:
		\[
		c_\alpha(t)=<\alpha ,(\Phi(u+th)-\Phi(u))>.
		\]
Of course, we have  $\dot{c}_{\alpha}(t)=<\alpha, T_{u+th}\Phi(h)>$.\\
Denote by $]u,v[$ the set of points $\{w=u+th, t\in ]0,1[\}$. As we have $c_\alpha(0)=c_\alpha(1)=0$, from Rolle's Theorem, there exists $u_\alpha\in ]u,v[$ such that
\begin{eqnarray}
\label{eq_AF2}
		<\alpha,\delta_l(T_{u_\alpha}\Phi(h))>=0
\end{eqnarray}
Note that, for any $t\in \mathbb{R}$,  this relation is also true  for any $ th$.
From our assumption, it follows that,  for each $k\in \mathbb{N}\setminus\{0\}$, there exists $u_k$ and $v_k$ in $B_{\mathbb{F}}(0,\frac{\eta}{k})$ so that $u_k\not =v_k$ but with $\Phi(u_k)=\Phi(v_k)$. So from the previous argument, for any  $\alpha\in\mathbb{M}^*$, we have

\begin{eqnarray}
\label{eq_AF2n}
		<\alpha,(T_{u_{\alpha,k}}\Phi(h_k))>=0
\end{eqnarray}
for some $u_{\alpha, k}\in]u_k,v_k[$ and  $h_k=v_k-u_k$.
 From (\ref{eq_AF2n}), for any  $i\in \mathbb{N}$, any $t\in \mathbb{R}$  and any  $ \alpha_i\in \mathbb{M}_i^*$, if $\alpha=\delta_i^*(\alpha_i)$, we have:
		\[
		|<\alpha,T_0\Phi(th_k)>|=|<\alpha_i,\delta_i([T_0\Phi-T_{u_{\alpha}}\Phi](th_k))>|=|<\alpha_i, [T_0\Phi_i-T_{\delta_i(u_{\alpha})}\Phi_i](\lambda_i(th_k))|
		\]
Denote by $\|.\|^{\mathbb{M}_i^*}$ the canonical norm on $\mathbb{M}_i^*$ associated to $\|.\|^{\mathbb{M}_i}$. Then from (\ref{eq_AF2n}) and since $u\mapsto T_u\Phi$ is $K_1$-Lipschitz on $B'(0,\eta)$ (for some constant $K_1$) we obtain:	
\begin{eqnarray}
\label{eq_InegalphaT0phi}
\begin{aligned}
|<\alpha,T_0\Phi(th_k)>|  &\leq  (||\alpha_i||^{\mathbb{M}_i^*}.K_1. ||\delta_i(u_{\alpha})||^{\mathbb{M}_i}||.||\lambda_i(th_k)||^{\mathbb{E}_i} \\
		 &\leq  (||\alpha_i ||^{\mathbb{M}_i^*}. K_1.||\lambda_i(th_k)||^{\mathbb{E}_i}\displaystyle\frac{ \eta}{k}. 
		 \end{aligned}
\end{eqnarray}
Since $u_k\not=v_k$, there must exist  at least  one integer $i\in \mathbb{N}$ such that $\lambda_i(h_k)\not=0$. Thus, by taking $t=\displaystyle\frac{u_k-v_k}{{\nu}_i((u_k-v_k))}$  in (\ref{eq_InegalphaT0phi}), we may assume $t=1$ and $\|\lambda_i(h_k)\|^{\mathbb{E}_i}=1$
In this way,  for this choice of  $h_k$,  we have a $1$-form  $\bar{\beta}_{k,i}$ on the linear  space generated by $\lambda_i(h_k)$ in $\mathbb{F}'_i$ such that
		$<\bar{\beta}_{k,i},\lambda_i(h_k)>=1$ and with $(\|\bar{\beta}_{k,i}\|^{\mathbb{E}_i})^*=1$. From the Hahn-Banach Theorem,  we can extend this linear form to a form $\beta_{k,i}\in \mathbb{M}_i^*$ such that $<\beta_{k,i},\lambda_i(h_k)>=1$ and $(\|\beta_{k,i}\|^{\mathbb{E}_i})^*=1$.  But since each $T_0\Phi$ is injective, this implies that $T_0\Phi_i$ is injective and so the adjoint $T^*_0\Phi_i$ is surjective. This implies that there exists $\alpha_{k,i}\in \mathbb{M}_i^*$ such that $T_0^*\Phi_i(\alpha_{k,i})=\beta_{k,i}$. Thus from (\ref{eq_InegalphaT0phi}) we obtain
\begin{eqnarray}\label{eq_inequality1Alpha}
\begin{matrix}
1=|<\beta_{k,i},\lambda_i(h_k)>|&=|<\alpha_{k,i}, T_0\Phi_i(\lambda_i(h_k))>|\hfill{}\\
                                                   &=|<\delta_i^*\alpha_{k,i}, T_0\Phi(h_k)>|\leq \|\alpha_{k,i}\|^{\mathbb{M}_i^*} K_1\displaystyle\frac{\eta}{k}\\
\end{matrix}
\end{eqnarray}
But, on the other hand since the operation "adjoint" is an isometry, from (\ref{eq_TPhiGlobalBounded}) we have
\begin{eqnarray}\label{eq_inequality2Alpha}
\displaystyle\frac{1}{M_1} \|\alpha_{k,i}\|^{\mathbb{M}_i^*} \leq \| T_0^*\Phi\alpha_{k,i}\|^{\mathbb{M}_i^*}=1
\end{eqnarray}
which gives a contradiction with (\ref{eq_inequality1Alpha}) for $k$ large enough.\\
\end{proof}
\bigskip

\appendix

\section{Projective limits}
\label{_ProjectiveLimits}

\subsection{Projective limits of topological spaces}
\label{__ProjectiveLimitsOfTopologicalSpaces}

\begin{definition}
\label{D_ProjectiveSequenceTopologicalSpaces}
A projective sequence of topological spaces\index{projective sequence!of topological spaces} is a sequence\\
 $\left( \left(  X_{i},\delta_{i}^{j}\right) \right)_{(i,j) \in \mathbb{N}^2,\; j \geq i}$ where

\begin{description}
\item[\textbf{(PSTS 1)}]
For all $i\in\mathbb{N},$ $X_{i}$ is a topological space;

\item[\textbf{(PSTS 2)}]
For all $\left( i,j \right)\in\mathbb{N}^2$ such that $j\geq i$,
$\delta_{i}^{j}:X_{j}\to X_{i}$ is a continuous map;

\item[\textbf{(PSTS 3)}]
For all $i\in\mathbb{N}$, $\delta_{i}^{i}={Id}_{X_{i}}$;

\item[\textbf{(PSTS 4)}]
For all $\left( i,j,k \right)\in\mathbb{N}^3$ such that $k \geq j \geq i$, $\delta_{i}^{j}\circ\delta_{j}^{k}=\delta_{i}^{k}$.
\end{description}
\end{definition}

\begin{notation}
\label{N_ProjectiveSequence}
For the sake of simplicity, the projective sequence $\left( \left(  X_{i},\delta_{i}^{j}\right) \right)_{(i,j) \in \mathbb{N}^2,\; j \geq i}$ will be denoted $\left(  X_{i},\delta_{i}^{j} \right) _{j\geq i}$.
\end{notation}

An element $\left(  x_{i}\right)  _{i\in\mathbb{N}}$ of the product
${\displaystyle\prod\limits_{i\in\mathbb{N}}}X_{i}$ is called a \emph{thread}\index{thread} if, for all $j\geq i$, $\delta_{i}^{j}\left(  x_{j}\right)=x_{i}$.

\begin{definition}
\label{D_ProjectiveLimitOfASequence}
The set $X=\underleftarrow{\lim}X_{i}$\index{$X=\underleftarrow{\lim}X_{i}$} of all threads, endowed with the finest topology for which all the projections $\delta_{i}:X\to X_{i} $ are continuous, is called the projective limit of the sequence\index{projective limit!of a sequence} $\left(  X_{i},\delta_{i}^{j} \right) _{j\geq i}$.
\end{definition}

A basis\index{basis!of a topology} of the topology of $X$ is constituted by the subsets $\left( \delta_{i} \right)  ^{-1}\left(  U_{i}\right)  $ where $U_{i}$ is an open subset of $X_{i}$ (and so $\delta_i$ is open whenever $\delta_i$ is surjective).

\begin{definition}
\label{D_ProjectiveSequenceMappings}
Let $\left(  X_{i},\delta_{i}^{j} \right)  _{j\geq i}$ and $\left(  Y_{i},\gamma_{i}^{j} \right)  _{j\geq i}$ be two projective sequences whose respective projective limits are $X$ and $Y$.

A sequence $\left(  f_{i}\right)  _{i\in\mathbb{N}}$ of continuous mappings $f_{i}:X_{i}\to Y_{i}$, satisfying, for all $(i,j) \in \mathbb{N}^2,$ $j \geq i,$ the coherence condition\index{coherence condition}
\[
\gamma_{i}^{j}\circ f_{j}=f_{i}\circ\delta_{i}^{j}
\]
is called a projective sequence of mappings\index{projective sequence!of mappings}.
\end{definition}

The projective limit of this sequence is the mapping
\[
\begin{array}
[c]{cccc}
f: & X & \to & Y\\
& \left(  x_{i}\right)  _{i\in\mathbb{N}} & \mapsto & \left(  f_{i}\left(
x_{i}\right)  \right)  _{i\in\mathbb{N}}
\end{array}
\]

The mapping $f$ is continuous if all the $f_{i}$ are continuous (cf. \cite{AbbMa}).

\subsection{Projective limits of Banach spaces}
\label{__ProjectiveLimitsOfBanachSpaces}
Consider a projective sequence $\left(  \mathbb{E}_{i},\delta_{i}^{j} \right)  _{j\geq i}$ of Banach spaces.
\begin{remark}
\label{R_ProjectiveSequenceOfBondingsMapsBetweenBanachSpacesDeterminedByConsecutiveRanks}
Since we have a countable sequence of Banach spaces, according to the properties of bonding maps, the sequence  $\left( \delta_i^j\right)_{(i,j)\in \mathbb{N}^2, \;j\geq i}$ is well defined by the sequence of bonding maps $\left( \delta_i^{i+1}\right) _{i\in \mathbb{N}}$.
\end{remark}

\subsection{Projective limits of differential maps}
\label{__ProjectiveLimitsOfDifferentialMapsBetweenFrechetSpaces}
The following proposition (cf. \cite{Gal1}, Lemma 1.2 and \cite{CaPe}, Chapter 4) is essential 
\begin{proposition} 
\label{P_ProjectiveLimitsOfDifferentialMaps}
Let $\left( \mathbb{E}_i,\delta_i^j \right) _{j\geq i}$ be a projective sequence of Banach spaces whose projective limit is the Fr\'echet space $\mathbb{F}=\underleftarrow{lim} \mathbb{E}_i$ and $ \left( f_i : \mathbb{E}_i \to \mathbb{E}_i  \right) _{i \in \mathbb{N}} $ a projective sequence of differential maps whose projective limit is $f=\underleftarrow{\lim} f_i$.
Then the following conditions hold:
\begin{enumerate}
\item
$f$ is smooth in the convenient sense (cf. \cite{KrMi})
\item
For all $x = \left( x_i \right) _{i \in \mathbb{N}}$, $df_x = \underleftarrow{\lim} { \left( df_i \right) }_{x_i} $.
\item
$df = \underleftarrow{\lim}df_i$.
\end{enumerate}
\end{proposition}

\subsection{Projective limits of Banach manifolds and Banach Lie groups}
\label{__ProjectiveLimitsOfBanachManifolds}

\begin{definition}\cite{Gal1}
\label{D_ProjectiveSequenceofBanachManifolds}
The projective sequence $\left( M_{i},\delta_{i}^{j} \right) _{j\geq i}$ is called \textit{projective sequence of Banach manifolds}\index{projective sequence!of Banach manifolds} if
\begin{description}
\item[\textbf{(PSBM 1)}]
$M_{i}$ is a manifold modelled on the Banach space $\mathbb{M}_{i}$;

\item[\textbf{(PSBM 2)}]
$\left(  \mathbb{M}_{i},\overline{\delta_{i}^{j}}\right) _{j\geq i}$ is a projective sequence of Banach spaces;

\item[\textbf{(PSBM 3)}]
For all $x=\left(  x_{i}\right)  \in M=\underleftarrow{\lim}M_{i}$, there exists a projective sequence of local
charts $\left(  U_{i},\xi_{i}\right)  _{i\in\mathbb{N}}$ such that
$x_{i}\in U_{i}$ where one has the relation
\[
\xi_{i}\circ\delta_{i}^{j}=\overline{\delta_{i}^{j}}\circ\varphi_{j};
\]

\item[\textbf{(PSBM 4)}]
 $U=\underleftarrow{\lim}U_{i}$ is a non empty open set in $M$.
\end{description}
\end{definition}

Under the assumptions \textbf{(PSBM 1)} and  \textbf{(PSBM 2)} in Definition \ref{D_ProjectiveSequenceofBanachManifolds}, the assumptions \textbf{(PSBM 3)}] and \textbf{(PSBM 4)}  around $x\in M$ is called \emph{the projective limit chart property} around $x\in M$ and  $(U=\underleftarrow{\lim}U_{i}, \phi=\underleftarrow{\lim}\phi_{i})$ is called a \emph{projective limit chart}.

The projective limit $M=\underleftarrow{\lim}M_{i}$ has a structure of Fr\'{e}chet manifold modelled on the Fr\'{e}chet space $\mathbb{M}
=\underleftarrow{\lim}\mathbb{M}_{i}$ and is called a \emph{$\mathsf{PLB}$-manifold}\index{$\mathsf{PLB}$-manifold}. The differentiable structure is defined \textit{via} the charts $\left(  U,\varphi\right)  $ where $\varphi
=\underleftarrow{\lim}\xi_{i}:U\to\left(  \xi_{i}\left(U_{i}\right)  \right) _{i \in \mathbb{N}}.$\\
$\varphi$ is a homeomorphism (projective limit of homeomorphisms) and the charts changings $\left(  \psi\circ
\varphi^{-1}\right)  _{|\varphi\left(  U\right)  }=\underleftarrow{\lim
}\left(  \left(  \psi_{i}\circ\left(  \xi_{i}\right)  ^{-1}\right)
_{|\xi_{i}\left(  U_{i}\right)  }\right)  $ between open sets of
Fr\'{e}chet spaces are smooth in the sense of convenient spaces.\\

\begin{definition}\cite{Gal1}
\label{D_ProjectiveLimitOfBanachLieGroups}
 $\left(  G_{i},\delta_{i}^{j} \right) _{j \geq i}  $
is a projective sequence of Banach-Lie groups
where $G_{i}$ is modelled on $\mathbb{G}_{i}$ if , for all $i\in\mathbb{N}$, there exists a chart $\left(  U_{i},\varphi_{i}\right)  $ centered at the unity $e_{i}\in G_{i}$ such that:
\begin{description}
\item[\textbf{(PLBLG 1)}]
$\forall(i,j) \in \mathbb{N}^2: j \geq i, \; \delta_{i}^{j}(U_{j})\subset U_{i}$;

\item[\textbf{(PLBLG 2)}]
$\forall(i,j) \in \mathbb{N}^2: j \geq i, \; \overline{\delta_{i}^{j}}\circ\varphi_{j}=\varphi_{j}\circ\delta_{i}^{j}$;

\item[\textbf{(PLBLG 3)}]
$\underleftarrow{\lim}\varphi_{i}(U_{i})$ is a non empty  open set of $\mathbb{G}$ and $\underleftarrow{\lim}U_{i}$ is open in $G$ according to the projective limit topology.
\end{description}
\end{definition}

 A projective sequence of Banach-Lie groups $\left(  G_{i},\delta_{i}^{j} \right) _{j \geq i}  $ is \emph{submersive} if each $ \delta_i^j$ is a surjective submersion.

\begin{theorem}\cite{Gal1}
\label{T_PLGLieGroup}
Let ${G}=\underleftarrow{\lim}{G}_{i}$ be a the projective limit of a    projective sequence of Banach-Lie groups $\left(  G_{i},\delta_{i}^{j} \right) _{j \geq i}  $. Then we have the following properties:
\begin{enumerate}
\item
$G$ is a Fr\'echet-Lie group.
\item
If $\mathbf{L}(G_i)$ is the Lie algebra of $G_i$ then $\mathbf{L}(G)=\underleftarrow{\lim}{\mathbf{L}G}_{i}$.
\item
If $\exp_{G_i}$ is the exponential map for $G_i$, then $\exp_G=\underleftarrow{\lim}\exp_{G_i}$ is the exponential map of the Fr\'echet-Lie group $G$.
\end{enumerate}
\end{theorem}

\subsection{Projective limits of Banach vector bundles }
\label{__ProjectiveLimitsOfBanachVectorBundles}

Let $\left(  M_{i},\delta_{i}^{j}\right)  _{j\geq i}$ be a projective sequence of Banach manifolds where each manifold $M_{i}$ is modeled on the Banach space $\mathbb{M}_{i}$.\\
For any integer $i$, let $ \left( E_{i},\pi_{i},M_{i} \right) $ be the Banach
vector bundle whose type fibre is the Banach vector space $\mathbb{E}_{i}$
where $\left(  \mathbb{E}_{i},\lambda_{i}^{j}\right)  _{j\geq i}$ is a projective sequence of Banach spaces.

\begin{definition}
\label{D_ProjectiveSequenceBanachVectorBundles}
$\left( (E_i,\pi_i,M_i),\left(\xi_i^j,\delta_i^j \right) \right) _{j \geq i}$, where $\xi_i^j:E_j \to E_i$ is a morphism of vector bundles, is called a projective sequence of Banach vector bundles\index{projective sequence!of Banach vector bundles} on the projective sequence of manifolds $\left(  M_{i},\delta_{i}^{j}\right)  _{j\geq i}$ if, for all $ \left( x_{i} \right) $, there exists a projective sequence of
trivializations $\left(  U_{i},\tau_{i}\right)  $ of $\left(  E_{i},\pi
_{i},M_{i}\right) $, where $\tau_{i}:\left(  \pi_{i}\right)  ^{-1}\left(
U_{i}\right)  \to U_{i}\times\mathbb{E}_{i}$ are local
diffeomorphisms, such that $x_{i}\in U_{i}$ (open in $M_{i}$) and where
$U=\underleftarrow{\lim}U_{i}$ is a non empty open set in $M$
 where, for all $(i,j) \in \mathbb{N}^2$ such that $j\geq i,$ we have the compatibility condition
\begin{description}
\item[(\textbf{PLBVB})]
$\left(  \delta_{i}^{j}\times\lambda_{i}^{j}\right)  \circ\tau_{j}=\tau_{i}\circ \xi_i^j$.
\end{description}
\end{definition}

With the previous notations,  $(U=\underleftarrow{\lim}U_{i}, \tau=\underleftarrow{\lim}\tau_i)$   is called a \emph{ projective bundle chart limit}\index{projective bundle chart limit}. The triple of  projective limit
 $(E=\underleftarrow{\lim}E_{i}, \pi=\underleftarrow{\lim}\pi_{i}, M=\underleftarrow{\lim}M_{i}))$ is called a \emph{projective limit of Banach bundles} or $\mathsf{PLB}$-bundle for short. \\

The following proposition generalizes the result of \cite{Gal2} about the projective limit of tangent bundles to Banach manifolds (cf. \cite{DGV} and \cite{CaPe}). 

\begin{proposition}
\label{P_ProjectiveLimitOfBanachVectorBundles}
Let $\left( (E_i,\pi_i,M_i),\left(\xi_i^j,\delta_i^j \right) \right)_{j \geq i}$ be a projective sequence of Banach vector bundles. \\
Then $\left(  \underleftarrow{\lim}E_i,\underleftarrow{\lim}\pi_i,\underleftarrow{\lim}M_i \right)  $ is a Fr\'{e}chet vector bundle.
\end{proposition}

\begin{notation}
\label{N_ProjectiveSequenceBundles}
For the sake of simplicity, the projective sequence $\left( (E_i,\pi_i,M_i),\left(\xi_i^j,\delta_i^j \right) \right)_{j \geq i}$  will be denoted $ (E_i,\pi_i,M_i) $.
\end{notation}

\begin{definition}
\label{D_StrongProjectiveLimitOfBanachBundle}  A sequence $\left(  E_{i},\pi_{i},M_{i}\right)
$ is called a submersive  projective sequence of Banach vector bundles if $\left( M_i,\delta_i^j)_{j\geq i}\right)$ is a submersive projective sequence of Banach manifolds and if around each $x\in M=\underleftarrow{\lim}M_i$, there exists a  projective limit chart bundle $(U=\underleftarrow{\lim}U_{i}, \tau=\underleftarrow{\lim}\tau_i)$ such that for all $i\in \mathbb{N}$, we have a decomposition $\mathbb{E}_{i+1}=\ker\bar{\lambda}_i^{i+1}\oplus \mathbb{E}'_i$ such that the condition \emph{(\textbf{PLBVB})} is true.
\end{definition}

The projective limit  $(E,\pi, M)$ of a   projective sequence of Banach vector bundles $\left(E_i,\pi_i, M_i\right)
$ is called a \emph{submersive projective  limit of Banach bundles} or \emph{submersive $\mathsf{PLB}$-bundle}\index{submersive $\mathsf{PLB}$-bundle} for short.\\

Now, we have the following result: 

\begin{proposition}
\label{P_StrongProjectiveLimitOfBanachBundle}
Let $\left(  E_{i},\pi_{i},M_{i}\right)$ be a submersive projective sequence of Banach bundles. Then, for each $i\in \mathbb{N}$, the maps $\delta_i:M\to M_i$ and  $\lambda_i: E\to E_i$ are  submersions.
\end{proposition}

\subsection{ Projective limit of Banach Lie algebroids}\label{_ProjectiveLimitBanachLieAlgebroidsFinslerMetric}
\begin{definition}\label{_BanachLieAlgebroid} Let $\pi:E\to M$ be a Banach bundle.
\begin{enumerate}
\item[(1)] an anchor is a vector  bundle morphism $\rho:E\to TM$ and  $(E,\rho)$ is called an anchored bundle
\item[(2)] An almost Lie bracket $[.,.]_E$ on an anchored bundle $E$ is a sheaf of antisymmetric bilinear maps
\[
\lbrack.,.]_{E_U}:\Gamma\left( E_U\right)  \times\Gamma\left( E_U\right)
\to\Gamma\left( E_U\right)
\]
 for any open set $U\subseteq M$ and which satisfies the following properties
\begin{enumerate}
\item [\textbf{(AL 1)}]
the Leibniz identity:\index{Leibniz identity}
\[
\forall\left(   \mathfrak{a}_{1}, \mathfrak{a}_{2}\right)  \in \Gamma \left( E_U \right)  ^{2}, \forall f \in C^{\infty}(M)
,\ [ \mathfrak{a}_{1},f \mathfrak{a}_{2}]_{E_U}=f.[ \mathfrak{a}_{1}, \mathfrak{a}_{2}]_{{E_U}}+df(\rho( \mathfrak{a}_{1})). \mathfrak{a}_{2}.
\]
\item[\textbf{(AL 2)}]
For any  open set $U\subseteq M$ and any $(\mathfrak{a}_1,\mathfrak{a}_2) \in \Gamma(E_U)^2$, the map
\[
(\mathfrak{a}_1,\mathfrak{a}_2)\mapsto [\mathfrak{a}_1,\mathfrak{a}_2]_{E_U}
\]
only depends on the $1$-jets of $\mathfrak{a}_1$ and $\mathfrak{a}_2$.
\end{enumerate}
\item[(3)] An anchored bundle $(E,\rho)$ provided with  an almost Lie bracket $[.,.]_E$ which satisfies the Jacobi identity
\[
[[ \mathfrak{a}_{1}, \mathfrak{a}_{2}]_{E},\mathfrak{a}_3]_E+[[ \mathfrak{a}_{2}, \mathfrak{a}_{3}]_{E},\mathfrak{a}_1]_E+[[ \mathfrak{a}_{3}, \mathfrak{a}_{1}]_{E},\mathfrak{a}_2]_E=0
\]
\[\forall\left(   \mathfrak{a}_{1}, \mathfrak{a}_{2}\mathfrak{a}_{3}\right)  \in \Gamma \left( E_U \right)  ^{3}\]
is called a Lie algebroid.\\
\end{enumerate}
\end{definition}

\begin{definition}
\label{D_ProjectiveSequenceofBanachLieAlgebroids}
$\left(  E_{i},\pi_{i},M_{i},\rho_{i}, [.,.]_{i} \right)$ is called a  submersive projective sequence of split  Lie algebroids  if
\begin{description}
\item[\textbf{(PSBLA 1)}] $\left(  E_{i},\xi_i^j\right)  _{j\geq i} $ is a submersive
projective sequence of Banach vector bundles ($\pi_{i}:E_{i}\to M_{i})_{i\in\mathbb{N}}$ over the 
projective sequence of manifolds $ \left(  M_{i},\delta_{i}^{j}\right) _{j\geq i}$;

\item[\textbf{(PSBLA 2)}]
For all $\left( i,j \right) \in\mathbb{N}^2$ such that $j\geq i$, one has
\[
\rho_{i}\circ \xi_i^j=T\delta_{i}^{j}\circ\rho_{j}
\]
\item[\textbf{(PSBLA 3)}] For all $\left( i,j \right) \in\mathbb{N}^2$ such that $j\geq i$, one has
\[
\xi_i^j([.,.]_{j})=[\xi_i^j(.),\xi_i^j(.)]_{i}
\]
\item[\textbf{(PSBLA 4)}] For all $i \in\mathbb{N}$  and $x_i\in M$  the kernel $\ker (\rho_i)_{x_i}$ is complemented in the fiber $E_{x_i}$.
\end{description}
\end{definition}

\begin{proposition}\label{P_ProjectiveLimtAlgebroid}(\cite{CaPe}) Let  $\left(  E_{i},\pi_{i},M_{i},\rho_{i}, [.,.]_{i} \right)$ be  a  submersive projective sequence of split  Lie algebroids. Then  
$\left(E:=\underleftarrow{\lim}  E_{i},\pi:=\underleftarrow{\lim}\pi_{i},M:=\underleftarrow{\lim}M_{i},\rho=\underleftarrow{\lim}\rho_{i} \right)$ is  Fr\'echet anchored bundle and $\Delta=\rho(E)$ is a closed distribution on $E$
\end{proposition}

\begin{remark} Under the assomptions of Proposition \ref{P_ProjectiveLimtAlgebroid} , unfortunately $ [.,.] = \underleftarrow{\lim}[.,.]_{i} $ does not define a Lie bracket on the set of all local sections of $(E,\pi, M)$ but only on section which are projective limit of section of $(E_i,\pi_i,M_i)$. Therefore $\left(  E,\pi,M,\rho, [.,.]\right)$ does not have a Fr\'echet Lie algebroid structure.
\end{remark}

\section{The Banach space $\mathcal{H}_b\left(  \mathbb{F}_{1},\mathbb{F}_{2}\right) $}
\label{__TheFrechetSpaceHF1F2-TheBanachSpaceHbF1F2}

Any Fr\'echet space $\mathbb{F}$ can be realized as the limit of a surjective projective sequence of Banach spaces
$ \left( \mathbb{B}_n,\lambda_n^m \right)_{m \geq n}$. Following \cite{DGV}, 2.3, we can identify $\mathbb{F}$ with the projective limit of the projective sequence
\[
\left( \widehat{\mathbb{B}}_n=\{x=(x_i)\in \prod_{0\leq i\leq n}\mathbb{B}_i:\;\forall j\geq i, \; x_i=\lambda_i^j(x_j)\;\;  \}, \; \widehat{\lambda}_n^m=({\lambda}_n^n,\dots,{\lambda}_n^m )\right )_{m\geq n}.
\]

We denote by $\lambda_n:\mathbb{F} \to \mathbb{B}_n $ and $\widehat{\lambda}_n : \mathbb{F} \to \widehat{\mathbb{B}}_n $ the canonical  surjective projections. Let $(||\;||_n)_{n \in \mathbb{N}}$ be a sequence where $||\;||_n$ is a norm on $\mathbb{B}_n$. In  this way,
\[
\widehat{||\;||}_n=\sup_{0\leq i\leq n}||\;||_i
\]
defines a norm on $ \widehat{\mathbb{B}}_n$. Then
\begin{eqnarray}
\label{eq_SemiNormOnFrechetAssociatedToNormOnBanach}
\hat{\nu}_n=\widehat{||\;||}_n\circ \widehat{\lambda}_n \textrm{ (resp. } \nu_n=||\;||_n\circ {\lambda}_n)
\end{eqnarray}
is the semi-norm on $\mathbb{F}$ associated to the  sequence $(\widehat{||\;||}_n)$ (resp.$(||\;||_n)$). Moreover,  we have $\hat{\nu}_n=\displaystyle\max_{0\leq i\leq n}\nu_i$ and  the topology of $\mathbb{F}$ is defined by $(\hat{\nu}_n)$ or $(\nu_n)$.\\

Let $(\mathbb{F}_{1},\nu^1_n)$ (resp. $(\mathbb{F}_{2},\nu^2_n)$)$\mathbb{\ }$ be a graded Fr\'{e}chet
space.

Recall that a linear map $L:\mathbb{F}_{1}\to \mathbb{F}_{2}$ is \emph{continuous}\index{continuous linear map}\index{linear map!continuous} if
\[
\forall n\in\mathbb{N},\exists k_{n}\in\mathbb{N},\exists C_{n}>0:\forall
x\in\mathbb{F}_{1},\nu_{2}^{n}\left(  L.x\right)  \leq C_{n}\nu_{1}^{k_{n}
}\left(  x\right).
\]
The space $\mathcal{L}\left(  \mathbb{F}_{1},\mathbb{F}_{2}\right)  $ of
continuous linear maps between both these Fr\'{e}chet spaces generally drops
out of the Fr\'{e}chet category. Indeed, $\mathcal{L}\left(  \mathbb{F}
_{1},\mathbb{F}_{2}\right)  $ is a Hausdorff locally convex topological vector
space whose topology is defined by the family of semi-norms $\left\{
p_{n,B}\right\}  $:
\[
p_{n,B}\left(  L\right)  =\displaystyle\sup_{x\in
	B}\left\{  \nu^{2}_{n}\left(  L.x\right) \right\}
\]
where $n\in\mathbb{N}$ and $B$ is any bounded subset of $\mathbb{F}_1$. This topology is not metrizable since the
family $\left\{  p_{n,B}\right\}  $ is not countable.\\
So $\mathcal{L}\left(  \mathbb{F}_{1},\mathbb{F}_{2}\right)  $ will be replaced,
under certain assumptions, by a projective limit of appropriate functional
spaces as introduced in \cite{Gal2}.

We denote by $\mathcal{L}\left(  \mathbb{B}_{1}^{n},\mathbb{B}_{2}^{n}\right) $ the space of linear continuous maps (or equivalently bounded
linear maps because $\mathbb{B}_{1}^{n}$ and $\mathbb{B}_{2}^{n}$ are normed
spaces). We then have the following result (\cite{DGV}, Theorem 2.3.10).

\begin{theorem}
\label{T_HF1F2}
The space of all continuous linear maps between $\mathbb{F}_{1}$ and $\mathbb{F}_{2}$ which can be represented as projective limits
\[
	\mathcal{H}\left(  \mathbb{F}_{1},\mathbb{F}_{2}\right)  =\left\{  \left(
	L_{n}\right)  \in\prod\limits_{n\in\mathbb{N}}\mathcal{L}\left(
	\mathbb{B}_{1}^{n},\mathbb{B}_{2}^{n}\right)  :\underleftarrow{\lim}
	L_{n}\text{ exists}\right\}
\]
\index{$\mathcal{H}\left(  \mathbb{F}_{1},\mathbb{F}_{2}\right)$}is a Fr\'{e}chet space.
\end{theorem}

For this sequence $\left(  L_{n}\right) _{n \in \mathbb{N}}  $ of linear maps, for any integer
$0\leq n\leq m$, the following diagram is commutative
\[
\xymatrix{
           \mathbb{B}_{1}^{n}  \ar@{<-}[r]^{\;\;\;{(\delta_1)}_{n}^{m}} \ar[d]_{L_n} & \mathbb{B}_{1}^{m} \ar[d]^{L_m}\\
           \mathbb{B}_{2}^{n} \ar@{<-}[r]^{\;\;\;{(\delta_2)}_{n}^{m}}               &  \mathbb{B}_{2}^{m}\\
}
\]
On $\mathcal{H}\left(  \mathbb{F}_{1},\mathbb{F}_{2}\right) $, the topology can be defined by the sequence of seminorms $p_n$ given by
\[
p_n\left(  L\right)  =\displaystyle\max_{0\leq k\leq n} \sup\left\{  \nu^{2}_{k}\left(  L.x\right)  ,x\in \mathbb{F}_1,\; \nu^1_k(x)=1
\right\}
\]
so that $\left(\mathcal{H}\left(  \mathbb{F}_{1},\mathbb{F}_{2}\right) ,p_n\right)$ is a graded Fr\'echet space.

\begin{remark}
	\label{R_SeminormEquivalentpn}
	 For $l\in\left\lbrace 1,2 \right\rbrace$ , given a graduation $ \left( \nu^l_n \right)$  on a Fr\'echet space $\mathbb{F}_l$, let $\mathbb{B}_l^n$ be  the associated local Banach space and $\delta_l^n:\mathbb{F}_l\to \mathbb{B}_l^n$ the canonical projection.\\
	 The quotient norm  $\tilde{\nu}^l_n$  associated to $\nu^l_n$ is defined by
	\begin{eqnarray}
	\label{eq_DefTildenu}
	\tilde{\nu}^l_n(\delta_n(z))=\sup\{\nu^l_n(y):\; \delta_n(y)=\delta_n(z) \}.
	\end{eqnarray}
	We denote by $(\tilde{\nu}^2_n)^{\operatorname{op}}$ the corresponding operator norm on $\mathcal{L}(\mathbb{B}_1^n,\mathbb{B}_2^n)$.\\
	If $L= \underleftarrow{\lim}L_n$ where $L_n:\mathbb{B}_1^n\to\mathbb{B}_2^n$, then we have
	\[
	(\tilde{\nu}^2_n)^{\operatorname{op}}(L)=\sup\{\tilde{\nu}^2_n(L_n.x),\;\;x\in \mathbb{B}_1^n\;\; \tilde{\nu}^1_n(x)\leq 1\}=\sup\{\nu^2_n(L.x), x\in \mathbb{F}_1, \nu^1(x)\leq 1\}.
	\]
	This implies that
	\[
	p_n(L)=\displaystyle\max_{0\leq i\leq n}(\tilde{\nu}^2_i)^{\operatorname{op}}(L_n).
	\]
\end{remark}

\begin{definition}
	\label{D_UniformlyBoundedOperator}
	Let $(\mathbb{F}_{1},\nu^1_n)$ and  $(\mathbb{F}_{2},\nu^2_n)$ be  graded Fr\'{e}chet spaces.
	A linear map $L:\mathbb{F}_{1}\to \mathbb{F}_{2}$ is called a uniformly bounded operator, if
	\[
	\exists C>O : \forall n\in \mathbb{N}, \; \nu_n(L(x))\leq C\mu_n(x).
	\] 	
\end{definition}

We denote by $\mathcal{H}_b\left(  \mathbb{F}_{1},\mathbb{F}_{2}\right) $\index{$\mathcal{H}_b\left(  \mathbb{F}_{1},\mathbb{F}_{2}\right) $} the set of uniformly bounded operators. Of course $\mathcal{H}_b\left(  \mathbb{F}_{1},\mathbb{F}_{2}\right) $ is contained in $\mathcal{H}\left(  \mathbb{F}_{1},\mathbb{F}_{2}\right)$ and $L\in\mathcal{H}\left(  \mathbb{F}_{1},\mathbb{F}_{2}\right)$  belongs to $\mathcal{H}_b\left(  \mathbb{F}_{1},\mathbb{F}_{2}\right) $ if and only if $||L||_\infty:=\displaystyle\sup_{n\in \mathbb{N}} p_n(L)<\infty$ and so
\[
\mathcal{H}_b\left(  \mathbb{F}_{1},\mathbb{F}_{2}\right) =\left[\mathcal{H}\left(  \mathbb{F}_{1},\mathbb{F}_{2}\right)\right]_b:=\{L\in \mathcal{H}\left(  \mathbb{F}_{1},\mathbb{F}_{2}\right):\; ||L||_\infty <\infty\}
\]
When $\mathbb{F}=\mathbb{F}_1=\mathbb{F}_2$ and $\nu^1_n=\nu^2_n$ for all $n\in \mathbb{N}$, the set $\mathcal{H}\left(  \mathbb{F},\mathbb{F}\right)$ (resp. $\mathcal{H}_b\left(  \mathbb{F},\mathbb{F}\right))$ is simply denoted $\mathcal{H}\left(  \mathbb{F}\right)$ (resp.  $\mathcal{H}_b\left(  \mathbb{F}\right)$). \\

We denote by $\mathcal{IH}_{b}\left(  \mathbb{F}_{1},\mathbb{F}_{2}\right) $\index{$\mathcal{IH}_{b}\left(  \mathbb{F}_{1},\mathbb{F}_{2}\right) $}  (resp. $\mathcal{SH}_{b}\left(  \mathbb{F}_{1},\mathbb{F}_{2}\right) $\index{$\mathcal{SH}_{b}\left(  \mathbb{F}_{1},\mathbb{F}_{2}\right) $}) the set of injective  (resp. surjective)  operators of $\mathcal{H}_b\left(  \mathbb{F}_{1},\mathbb{F}_{2}\right) $ with closed range.
\begin{proposition} (\cite{CaPe})
\label{P_InjectiveSurjectiveBH}
${}$
\begin{enumerate}
		\item
 Each operator $L\in \mathcal{H}\left(  \mathbb{F}_{1},\mathbb{F}_{2}\right) $ has a closed range if and only if,  for each $n\in \mathbb{N}$, the induced operator $L_n:\mathbb{B}_1^n\to\mathbb{B}_2^n$ has  a closed range.
		\item
 $\mathcal{IH}_{b}\left(  \mathbb{F}_{1},\mathbb{F}_{2}\right) $ is an open subset of  $\mathcal{H}_b\left(  \mathbb{F}_{1},\mathbb{F}_{2}\right) $.
		\item
 $\mathcal{SH}_{b}\left(  \mathbb{F}_{1},\mathbb{F}_{2}\right) $ is an open subset of  $\mathcal{H}_b\left(  \mathbb{F}_{1},\mathbb{F}_{2}\right) $.
\end{enumerate}
\end{proposition}

 We will give the sketch  of the proof of Point (2) since some arguments used in this proof are also useful for the proof of Theorem \ref{T_IntegrabilitySPLBanachBundles}:\\

 \begin{proof}

(2) Consider an  injective  operator $L\in \mathcal{H}\left(\mathbb{F}_1,\mathbb{F}_2\right)$. According to the representation $\mathbb{F}_i=\underleftarrow{\lim} \mathbb{B}_i^n$ as a projective limit  of a
	projective Banach sequence $\left( \mathbb{B}_i^n,(\delta_i)_n^m\right) _{m \geq n} $, we have a sequence of linear operators $L_n: \mathbb{B}_1^n\to \mathbb{B}_2^n$	such that $L=\underleftarrow{\lim}L_n$ (cf. Theorem \ref{T_HF1F2}). Considering each
 \[
 \mathbb{F}_i=\{ (x_n)\in \displaystyle\prod_{n\in \mathbb{N}} \mathbb{B}_i^n:\;\forall m\geq n,\; x_n=(\delta_2)_n^m(x_m) \}
 \]
, then if $x=(x_n)\in \mathbb{F}_1$ then $L(x)=(L_n(x_n))\in \mathbb{F}_2$. Thus it is clear that $L$ is injective if and only if  $L_n$ is injective for all $n\in \mathbb{N}$. \\

Now if $L\in \mathcal{IH}_{b}\left( \mathbb{F}_{1},\mathbb{F}_{2}\right) $,  then $L_n$ is an isomorphism from $\mathbb{B}_1^n$  onto  its range and so we have
	\begin{eqnarray}
	\label{eq_Lniso}
		\frac{1}{\ell_n}.\tilde{\nu}^1_n(x)\leq \tilde{\nu}^2_n(L_n(x)\leq \ell_n.\tilde{\nu}^1_n(x)
	\end{eqnarray}
	for all $x\in \mathbb{B}_1^n$, all $n\in \mathbb{N}$, where $\tilde{\nu}^i_n$ is the quotient norm of $\nu^i_n$ on $\mathbb{B}_i^n$  for $i \in {1,2}$,  and
	\[
	\ell_n=(\tilde{\nu}^2_n)^{\operatorname{op}}(L)=\sup\{\displaystyle\frac{\tilde{\nu}^2_n(L_n(x))}{\tilde{\nu}^1_n(x)}\;:\; x\not=0\}.
	\]
	Since $\delta^2_n$ is the canonical projection of $\mathbb{F}_2$ on $\mathbb{B}_2^n$ and $\nu^2_n\circ  \delta_n=\tilde{\nu}^2_n$, we obtain
	\begin{eqnarray}\label{eq_Liso}
		\frac{1}{\ell_n}{\nu}^1_n(x)\leq {\nu}^2_n(L(x))\leq \ell_n.{\nu}^1_n(x)
	\end{eqnarray}
	for all $x\in \mathbb{F}_1$ and $n\in \mathbb{N}$. But we have $\ell_n\leq ||L||_\infty$ and we finally obtain
	\begin{eqnarray}
	\label{eq_Lisobounded}
		\frac{1}{\ell}{\nu}^1_n(x)\leq {\nu}^2_n(L(x))\leq \ell.{\nu}^1_n(x)
	\end{eqnarray}
	for all $x\in \mathbb{F}_1$, all $n\in \mathbb{N}$ and where $\ell=||L||_\infty$.
	
	Fix some $L\in \mathcal{H}_b\left(  \mathbb{F}_{1},\mathbb{F}_{2}\right) $ and set $\ell=||L||_\infty$, we consider the open set
	$$W=\{T\in \mathcal{H}_b\left(  \mathbb{F}_{1},\mathbb{F}_{2}\right),\;:\; ||T-L||_\infty<\frac{\ell}{2}\}$$
	Fix some $n\in \mathbb{N}$. For any $x\in \mathbb{F}_1$  and $T\in W$, we have
	\[
	\nu^1_n(x)-\nu^1_n(T(x))\leq \nu^1_n(T- L)(x)\leq p_n(T-L).\nu^1_n(x)\leq ||T-L||_\infty.\nu^1_n(x)\leq \frac{\ell}{2}\nu^1_n(x).	
	\]
	This implies that
	\begin{eqnarray}
	\label{eq_injectiveT}
		\nu^2_n(T(x))\geq \frac{\ell}{2}\nu^1_n(x).
	\end{eqnarray}
	Since  $(\nu^i_n)$ is a separating sequence of semi-norms, it follows that $L$ is injective. \\
	Now taking in account inequality (\ref{eq_injectiveT}) and relation $\tilde{\nu}^i_n=\nu^i_n\circ (\delta_i)_n$, for $T\in W$ and  each $n\in \mathbb{N}$, we have
	\[
	\tilde{\nu}^2_n(T_n(x))\leq \displaystyle\frac{3 \ell}{2} \nu^1_n(x)\leq 3 \tilde{\nu}^2_n(T_n(x))
	\]	
	for all $x\in \mathbb{B}_1^n$. It follows that $T_n$ is closed and so $T$ is closed (cf. 1.).
	Finally,  $W$  is an open neighbourhood of $L$ contained in $\mathcal{IH}_{b}\left(  \mathbb{F}_{1},\mathbb{F}_{2}\right) $, which ends the proof of (2).\\
	
\end{proof}

From this Proposition we have

\begin{theorem}(\cite{CaPe})
\label{T_UniformlyBoundedProperties}
${}$
\begin{enumerate}
\item
The Banach space $\mathcal{H}_b(\mathbb{F})$ has a Banach-Lie algebra structure and the set $\mathcal{GH}_b(\mathbb{F})$ of uniformly bounded isomorphisms of $\mathbb{F}$ is open in $\mathcal{H}_{b}(  \mathbb{F})$.
\item
$\mathcal{GH}_b(\mathbb{F})$ has a structure of  Banach-Lie group whose Lie algebra is $\mathcal{H}_b(\mathbb{F})$.
\item
If $\mathbb{F}$ is identified with the projective $\underleftarrow{\lim} \mathbb{B}^n$ we denote by $\exp_n:\mathcal{L}(\mathbb{B}_n) \to \mathcal{GL}(\mathbb{B}_n)$, then we a have a well defined smooth map $\exp:=\underleftarrow{\lim}\exp_n: \mathcal{H}_b(\mathcal{F}) \to \mathcal{GH}_b(\mathbb{F})$ which is a diffeomorphism from an open set of $0 \in \mathcal{H}_b(\mathcal{F})$ onto a a neighbourhood of $\operatorname{Id}_\mathbb{F}$.
\end{enumerate}
\end{theorem}

\section{A theorem of existence of ODE}\label{ExistenceODE}

 The following result is in fact a reformulation in our context of Theorem 1 in \cite{Lob}.  
 
 \begin{theorem}
	\label{T_SolutionODEOnFrechetSpaces}
	Let $\mathbb{F}$ a Fr\'echet space realized as the limit of a surjective projective sequence of Banach spaces
	$ \left( \mathbb{B}_n,\lambda_n^m \right)_{m \geq n}$ whose topology is defined by the sequence of seminorms $\left( \nu_n\right)_{n \in \mathbb{N}}$.  Let $I$ be an open interval in $\mathbb{R}$ and $U$ be an open set of $I\times \mathbb{F}$. Then  $U$ is a surjective projective limit of open sets $U_n \subset I \times \mathbb{B}_n$. Consider a smooth map $f=\underleftarrow{\lim}f_n: U \to \mathbb{F}$, projective limit of maps $f_n: U_n\to \mathbb{B}_n$.
\footnote{This means that we have: $ \forall m\geq n, \; \lambda_n^m\circ f_m=f_n\circ (Id_\mathbb{R}\times \lambda_n^m)$}
	Assume that  for every point $\left( t,x\right)\in U$, and every $n\in \mathbb{N} $, there exists an integrable function $K_n>0$ such that
	\begin{equation}
		\label{eq_LipschitzAssumption}
		\forall \left( (t,x), (t,x')\right) \in U^2, \; \nu_n(f(t,x)-f(t,x'))\leq K_n(t)\nu_n(x-x').
	\end{equation}
	and consider the differential equation:
	\begin{equation}
		\label{eq_ODEInFrechetSpace}
		\dot{x}=\phi \left( t,x \right).
	\end{equation}
	\begin{enumerate}
		\item
For any $(t_0,x_0)\in U$, there exists  $\alpha >0$  with $I_\alpha=[t_0-\alpha, t_0+\alpha]\subset I$, an open pseudo-ball $V=B(x_0,r)\subset U$ and a map $\Phi: I_\alpha\times I_\alpha \times V\to \mathbb{F}$ such that
		\[
		t\mapsto \Phi(t,\tau, x)
		\]
		is the unique solution of (\ref{eq_ODEInFrechetSpace})
		with initial condition $\Phi(\tau,\tau, x )=x$ for all $x\in V$.
		\item
$V$ is the  projective limit of the open  balls  $V_n$ of $\mathbb{B}_n$.  For each $n\in \mathbb{N}$, the curve $t\mapsto \lambda_n\circ \Phi(t,\tau, \lambda_n(x))$ is the unique solution $\gamma:I_\alpha\to \mathbb{B}_n$ of the differential equation $\dot{x}_n=\phi_n \left( t,x_n \right)$ with initial condition $\gamma(\tau)=\lambda_n(x)$.\\
	\end{enumerate}
\end{theorem}
 From this theorem we obtain easily:\\
\begin{corollary}
	\label{C_ODEVectorFields}
	Let $U=\underleftarrow{\lim}U_n$ be an open subset of $\mathbb{F}$ and  $X=\underleftarrow{\lim}X_n: U \to \mathbb{F}$ a projective limit of smooth maps $X_n: U_n\to \mathbb{B}_n$. Assume that for every $n\in \mathbb{N}$ we have
	\begin{equation}
		\label{eq_LipschitzAssumptionCor}
		\forall \left( (t,x),(t,x')\right) \in U^2, \;\nu_n(X(x)-X(x'))\leq K_n \nu_n(x-x').
	\end{equation}
	For $x_0\in U$, let $B(x_0,2r)=\{x\in \mathbb{F},\;:\; \nu_{n_i}(x-x_0)<2r, \; 1\leq i\leq k\}$ be a pseudo-ball contained in $U$. Let us set
	\begin{description}
		\item $C_1=\max_{1\leq i\leq k} {K_{n_i}}$
		\item $C_2=\displaystyle\sup_{z \in B(x_0,r)}\left\{\max_{1\leq i\leq k}\nu_{n_i}(f(z))\right\}$.
	\end{description}
Then for any $\alpha>0$ such that $\alpha e^{2\alpha C_1}\leq \frac{r}{2C_2}$,  there
	exists a neighbourhood $V=B(x_0,r)$  and a smooth map $\phi_\alpha:I_\alpha\times V$ such that $t\mapsto \phi^\alpha(t,x)$ is the unique solution of $\dot{x}=X(x)$ defined on $I_\alpha$ with initial condition $\phi^\alpha(0,x)=x$. Moreover if $V_n=\lambda_n(V)$, consider $\phi_n^\alpha:I_\alpha\times V_n\to \mathbb{B}_n$ defined by $\phi^\alpha_n=\lambda_n\circ \phi^\alpha$; For each $z\in V_n$, the map  $t\mapsto \phi^\alpha(t,z)$ is the unique solution of the differential equation $\dot{x}_n=X_n(x_n)$  defined on $I_\alpha$ with initial condition  $\phi^{\alpha}(0, z)=z$
\end{corollary}

\begin{remark}
	\label{R_OneParameterGroup}
	If  $X=\underleftarrow{\lim}X_n$ is a  smooth vector field defined on an open set $U=\underleftarrow{\lim}U_n$ of $\mathbb{F}$, which satisfies assumption (\ref{eq_LipschitzAssumptionCor}), as classically,  according to Corollary \ref{C_ODEVectorFields}, the map $\operatorname{Fl}^X_t:=\operatorname{Fl}^X(t,\;)$ is \emph{the local flow of $X$}\index{local flow} that is $\operatorname{Fl}^X_t$ fullfils the  properties of a $1$-parameter group:
	\begin{description}
		\item $\operatorname{Fl}^X_0=Id_V$
		\item $\operatorname{Fl}^X_t\circ \operatorname{Fl}^X_s=\operatorname{Fl}^X_{s+t}$ if $s$,$t$ and $s+t$ belong to $I_\alpha$.	
	\end{description}
	In particular $\operatorname{Fl}^X_t$ is a diffeomorphism from $V$ onto it range and its inverse is $\operatorname{Fl}^X_{-t}$. Moreover $\operatorname{Fl}^{X_n}_t=\lambda_n\circ \operatorname{Fl}^X_t\circ \lambda_n$ is  local flow of $X_n=\lambda_n\circ X\circ \lambda_n$ and we have $\operatorname{Fl}^X_t=\underleftarrow{\lim}\operatorname{Fl}^{X_n}_t$.\\
\end{remark}

 \end{document}